\documentclass{siamltex}

\usepackage{amsfonts,latexsym}
\usepackage{epsfig}
\usepackage{amsmath}
\usepackage{graphicx}
\usepackage{amsmath}
\usepackage{color}
\usepackage{ulem}
\usepackage{algorithm,algorithmic}
\usepackage[english]{babel}

\title{A  block  tangential Lanczos method for model reduction of large-scale first and second order dynamical systems} 

\author{K. Jbilou\thanks{Universit\'e du Littoral, C\^ote d'Opale, batiment H. Poincarr\'e, 50 rue F. Buisson, 62280 Calais Cedex, France, email: jbilou@univ-littoral.fr} \and  Y. Kaouane$^*$}

\newcommand{\di}{\displaystyle}

\newcommand{\R}{\mathbb{R}}
\newcommand{\C}{\mathbb{C}}

\newtheorem{lem}{Lemma}[section]

\begin{document}
\maketitle


\begin{abstract}
In this paper, we present a new approach for model reduction of large scale first and second order dynamical systems  with multiple inputs and multiple outputs (MIMO). This approach is  based on the projection of the initial problem onto  tangential Krylov subspaces to  produce a  simpler  reduced-order model that approximates well the behaviour of the original model. We present an algorithm named: Adaptive Block Tangential  Lanczos-type (ABTL) algorithm. We give some algebraic properties and present some numerical experiences to show the effectiveness of the proposed algorithms.
\end{abstract}

\begin{keywords}
 {Krylov subspaces, Model reduction, Interpolation, Tangential directions.}
\end{keywords}


\section{Introduction} 

Consider a linear time-invariant multi-input, multi-output linear time independent (LTI) dynamical 
system described by the state-space equations 
\begin{equation}\label{eq001}
\Sigma:= \left\{\begin{array}{lllll}
\dot{x}(t)  & = & Ax(t) +  Bu(t)\\
  y(t)  & = & Cx(t),\\
 \end{array}\right.
\end{equation}
where $x(t) \in  \R^n$ denotes the state vector, $u(t)$ and $y(t)$ are
the  input and the output signal vectors, respectively. The 
matrix $A  \in \R^{n\times n}$ is assumed to be large and sparse and  $B,\ C^{T} \in \R^{n\times p}$. The transfer function  associated to the 
system in (\ref{eq001}) is given as
\begin{equation}\label{eq002}
H(\omega) := C(\omega I_n-A)^{-1}B \in \R^{p\times p}.
\end{equation}
The goal of our model reduction approach  consists in defining  two orthogonal matrices  $V_m$ and $W_m\in \R^{n\times m}$ (with $m \ll n$) to produce a much
smaller order system $\Sigma_m$ with the state-space form
\begin{equation}\label{eq003}
\Sigma_m: \left\{\begin{array}{lllll}
\dot{x}_m(t)  & = & A_mx_m(t)+ B_mu(t)  \\
  y_m(t)  & = & C_mx_m(t),  \\
 \end{array}\right.
\end{equation}
and its  transfer function is defined by
\begin{equation}\label{eq004}
H_m(\omega) := C_m( \omega I_m-A_m)^{-1}B_m \in \R^{p\times p},
\end{equation}
where $A_m =W_m^TAV_m\in \R^{m\times m}$, $B_m=W_m^TB\in \R^{m\times p}$ and $ C_m=CV_m\in \R^{p\times m}$, 
such that the reduced system $\Sigma_m$ will have an output $y_m(t)$ as close as possible to the one of  the original system to any
given input $u(t)$, which means that for some chosen norm, $\Vert y-y_m \Vert$ is small.\\

\
\\
Various model reduction  techniques, such as Pad\'e approximation \cite{B9,B30}, balanced truncation \cite{B22}, optimal Hankel norm \cite{B31} and Krylov subspace  methods, \cite{B32,B33,B34,B35} have been used for large multi-input multi-output (MIMO) dynamical systems, see  \cite{B23,B31,B36}. Balanced Truncation Model Reduction (BTMR) method  is a very popular
method; \cite{B004,B002}; the method preserves the stability and provides a bound for the approximation error. In the case of small to medium systems,
(BTMR) can be implemented efficiently. However, for large-scale settings, the method is quite expensive to implement, because it requires the computation of two Lyapunov equations, and results in a computational complexity of $\mathcal{O}(n^3)$ and a storage requirement of $\mathcal{O}(n^2)$, see \cite{B004,B006,B007}. In this paper, we project the system \eqref{eq001}  onto  the following block tangential Krylov subspaces  defined as,
   $$
  \widetilde{\mathcal{K}}_m(A,B)= Range\{(\sigma_1I_n-A)^{-1}BR_1,...,(\sigma_mI_n-A)^{-1}BR_m\},
$$
  $$
  \widetilde{\mathcal{K}}_m(A^T,C^T)= Range\{(\mu_1I_n-A)^{-T}C^TL_1,...,(\mu_mI_n-A)^{-T}C^TL_m\},
$$
in order to obtain a small scale dynamical system. The $\{\sigma_i\}_{i=1}^{m}$ and  $\{\mu_i\}_{i=1}^{m}$ are the right and left interpolation points, the  $\{R_i\}_{i=1}^{m}$ and $\{L_i\}_{i=1}^{m}$ are the right and left blocks tangent directions with $R_i,\ L_i\in \R^{p\times s}$ with $s \leq p$. Later, we will show how to choose these tangent interpolation points and directions.  \\

\noindent The paper is organized as follows: In section 2 we give some definition used later and we introduce the tangential interpolation. In  Section 3, we present the tangential block Lanczos-type method and the corresponding algorithm. Section 4 is devoted to the selection of the interpolation points and the tangential directions that are used in the construction of  block tangential Krylov subspaces, and we present briefly the adaptive tangential block  Lanczos-type algorithm. Section 5 we treat the model reduction of second-order systems. The last section is devoted to  numerical tests and comparisons with some well known model order reduction methods.
 
\section{Moments and interpolation}
We first give the following definition.
\begin{definition}
Given the  system $\Sigma$, its associated transfer 
function $H(s) =C(\omega I_n-A)^{-1}B $ can be decomposed through a
 Laurent series expansion around a given $\sigma \in \R$ 
 (shift point), as follows
 \begin{equation}\label{eq005}
H(\omega)= \sum_{i=0}^\infty \eta_i^{(\sigma)}  \frac{(\omega-\sigma)^i}{i!},
\end{equation}
where $\eta_i^{(\sigma)}  \in \R^{p\times p}$ is called the $i$-th moments at $\sigma$ associated to the system and  defined as follows
\begin{equation}\label{eq006}
\eta_i^{(\sigma)} = C(\sigma I_n-A)^{-(i+1)}B =(-1)^i\frac{d^i}{d\omega^i}H(\omega)|_{\omega=\sigma},\qquad i=0,1,...
\end{equation}
In the case where $ \sigma= \infty$ the moments are called Markov
parameters and are given by $$\eta_i = CA^iB.$$
\end{definition}

\noindent\textbf{Problem:} Given a full-order model (\ref{eq001}) and assume that the following parameters are given:
\begin{itemize}
\item Left interpolation points $\{\mu_i\}^m_{i=1} \subset \C$  \&  block tangent directions $\{L_i\}^m_{i=1} \subset \C^{p\times s}$.
\item Right interpolation points $\{\sigma_i\}^m_{i=1} \subset \C$ \&  block tangent directions $\{R_i\}^m_{i=1} \subset \C^{p \times s}$.
\end{itemize}
\vspace{0.2cm}
\noindent The main problem is to find a reduced-order model (\ref{eq003})
  such that the associated transfer function, $H_m$
in (\ref{eq004}) is a tangential interpolant to $H$ in (\ref{eq002}), i.e.

\begin{equation}\label{eq007}
 \left\{\begin{array}{lllll}
H_m(\sigma_i)R_i  & = & H(\sigma_i)R_i  \\
 & & & for \ i=1,...,m.\\
 L_i^TH_m(\mu_i)   & = & L_i^TH_(\mu_i)  \\
 \end{array}\right.
\end{equation}

\noindent The interpolation points and tangent directions are selected to realize the model
reduction goals described later.

\section{The  block tangential Lanczos-type method}\quad

\noindent  Let the original transfer function  $H(\omega) = C(\omega I-A)^{-1}B$ be expressed as 
$H(\omega)=CX=Y^TB$ where $X$ and $Y$ are such that,
\begin{equation}\label{eq11}
(\omega I_n-A)X =B, \;\; {\rm and }\;\; 
(\omega I_n-A)^TY =C^T.
\end{equation}
Given a system of matrices $\{V_1,\ldots,V_m\}$ and $\{W_1,\ldots,W_m\}$ where $V_i, \ W_i\in \R^{n\times s}$,  the approximate solution $X_m$ and $Y_m$ of $X$ and $Y$ are computed such that
\begin{equation}
\label{ex}
X_m^i\in Range\{V_1,...,V_m\} \ \ and \ \  Y_m^i\in Range\{W_1,...,W_m\}
\end{equation}
and
\begin{equation}
\label{er}
R_B^i(\omega)\bot\ Range\left\{W_1,...,W_m\right\},\,\,i=1,...,p
\end{equation}
\begin{equation}
\label{er1}
R_C^i(\omega) \bot\ Range\left\{V_1,...,V_m\right\},\,\,i=1,...,p
\end{equation}
where $X_m^i$, $Y_m^i$, $R_B^i$ and  $R_C^i$ are  the  $i$-th columns of $X_m$, $Y_m$, $R_B=B-(\omega I_n-A)X_m$ and  $R_C=C^T-(\omega I_n-A)^TY_m$, respectively. 
If we set $\mathbb{V}_m=[V_1,\ldots,V_m]$ and  $\mathbb{W}_m=[W_1,\ldots,W_m]$, then from \eqref{ex}, \eqref{er} and \eqref{er1} , we obtain 
$$X_m=\mathbb{V}_m(\omega I_{ms}-A_m)^{-1}\mathbb{W}_m^TB, $$
$$Y_m=\mathbb{W}_m(\omega I_{ms}-A_m)^{-T}\mathbb{V}_m^TC^T, $$
which gives the following approximate transfer function
$$H_m(\omega) =C_m(\omega I_{ms}-A_m)^{-1}B_m,$$
where $A_m=\mathbb{W}_m^TA\mathbb{V}_m$, $B_m=\mathbb{W}_m^TB$ and $C_m=C\mathbb{V}_m$. 
The matrices $\mathbb{V}_m=[V_1,...,V_m]$ and $\mathbb{W}_m=[W_1,...,W_m]$ are bi-orthonormal, where the $V_i, \ W_i\in \R^{n\times s}$. 
Notice that the residuals can be expressed as 
 \begin{equation}\label{eq12}
 R_B(\omega)=B-(\omega I_n -A)\mathbb{V}_m(\omega I_{ms}-A_m)^{-1}\mathbb{W}_m^TB.
 \end{equation}
\begin{equation}\label{eq13}
 R_C(\omega)=C^T-(\omega I_n -A)^T\mathbb{W}_m(\omega I_{ms}-A_m)^{-T}\mathbb{V}_m^TC^T.
 \end{equation}

\subsection{Tangential block Lanczos-type algorithm}\quad

\noindent This algorithm consists in constructing two bi-orthonormal bases,  spanned by the columns of $\{V_1,V_2,\ldots,V_m\}$ and $\{W_1,W_2,\ldots,W_m\}$,  of the following block tangential Krylov subspaces  \begin{equation}\label{eq019}
\mathcal{K}_m(A,B)=Range\{(\sigma_1I_n-A)^{-1}BR_1,...,(\sigma_mI_n-A)^{-1}BR_m\},
\end{equation}
and
\begin{equation}\label{eq0190}
\widetilde{\mathcal{K}}_m(A^T,C^T)=Range\{(\mu_1I_n-A)^{-T}C^TL_1,...,(\mu_mI_n-A)^{-T}C^TL_m\},
\end{equation}
where $\{\sigma_i\}_{i=1}^{m}$ and $\{\mu_i\}_{i=1}^{m}$ are the right and left interpolation points respectively and  $\{R_i\}_{i=1}^{m}$, $\{L_i\}_{i=1}^{m}$ are the right and left tangent directions with $R_i,\  L_i\in \R^{p\times s}$.  We present next the  Block Tangential Lanczos (BTL) algorithm that allows us to construct such bases. 
It is summarized in the following steps.

\begin{algorithm}[h!]
\caption{ Block Tangential  Lanczos (BTL) algorithm }\label{BTL}
		-- Inputs: A, B, C, $\sigma = \{\sigma_i\}_{i=1}^{m+1}$, $\mu = \{\mu_i\}_{i=1}^{m+1}$,  $R = \{R_i\}_{i=1}^{m+1}$, $L = \{L_i\}_{i=1}^{m+1}$,  $R_i,\ L_i\in \R^{p\times s}$.\\
		-- Output: $\mathbb{V}_{m+1} = \left[V_1,...,V_{m+1}\right],$ $\mathbb{W}_{m+1} = \left[W_1,...,W_{m+1}\right].$
\begin{itemize}
\item Compute $(\sigma_1 I_n-A)^{-1}BR_1 = H_{1,0}V_1$ and $(\mu_{1}I_n-A)^{-T}C^TL_{1} = F_{1,0}W_1$ such that $W_1^TV_1=I_p$.
\item Initialize: $\mathbb{V}_1= [V_1]$, $\mathbb{W}_1= [W_1]$.
\item For j = 1,...,m
\begin{enumerate}
\item If \quad$\sigma_{j+1} \ne  \infty$, $\widetilde{V}_{j+1}=(\sigma_{j+1}I_n-A)^{-1}BR_{j+1}$,  else $\widetilde{V}_{j+1}= AB R_{j+1}$.
\item If \quad$\mu_{j+1} \ne  \infty$, $\widetilde{W}_{j+1}=(\mu_{j+1}I_n-A)^{-T}C^TL_{j+1}$,  else $\widetilde{W}_{j+1}= AC^TL_{j+1}$.
\item   For i = 1,...,j
\begin{itemize}
\item  $H_{i,j} = W_i^T\widetilde{V}_{j+1}$,\qquad \qquad  \textbf{--} $F_{i,j} = V_i^T\widetilde{W}_{j+1}$,
\item  $\widetilde{V}_{j+1} = \widetilde{V}_{j+1} - V_i H_{i,j}$, \quad \;\textbf{--} $\widetilde{W}_{j+1} = \widetilde{W}_{j+1} - W_i F_{i,j}$,
\end{itemize}
\item End.
\item $\widetilde{V}_{j+1}=V_{j+1}H_{j+1,j}$,\ $\widetilde{W}_{j+1}=W_{j+1}F_{j+1,j}$, \quad (QR decomposition).
\item $W_{j+1}^TV_{j+1}=P_{j+1}D_{j+1}Q_{j+1}^T$, \quad (Singular Value Decomposition).
\item $V_{j+1}=V_{j+1}Q_{j+1}D_{j+1}^{-1/2}$, \  $W_{j+1}=W_{j+1}P_{j+1}D_{j+1}^{-1/2}$.
\item $H_{j+1,j}=D_{j+1}^{1/2}Q_{j+1}^TH_{j+1,j}$, \ $F_{j+1,j}=D_{j+1}^{1/2}P_{j+1}^TF_{j+1,j}$.
\item  $\mathbb{V}_{j+1}=\left[\mathbb{V}_j , \; V_{j+1}\right]$,\qquad $\mathbb{W}_{j+1}=\left[\mathbb{W}_j , \; W_{j+1}\right]$.
\end{enumerate}
\item End
\end{itemize}
\end{algorithm}

\noindent Let $\mathbb{V}_m=[V_1,V_2,\ldots,V_m]$ and $\mathbb{W}_m=[W_1,W_2,\ldots,W_m]$. Then, we should have the bi-orthogonality  conditions for $i,j=1,\ldots,m$:
\begin{equation}\label{eq020-}
\left\{\begin{array}{llll}
W_i^TV_j=I,  & i=j,\\
W_i^TV_j=0, & i\neq j.\\
 \end{array}\right. 
\end{equation} 
\noindent Here we suppose that we already have the set of  interpolation points $\sigma = \{\sigma_i\}_{i=1}^{m+1}$, $\mu = \{\mu_i\}_{i=1}^{m+1}$ and the tangential matrix directions $R = \{R_i\}_{i=1}^{m+1}$ and  $L = \{L_i\}_{i=1}^{m+1}$.  The upper block upper Hessenberg matrices $\widetilde{\mathbb{H}}_m =\left[\widetilde{\mathbb{H}}^{(1)},...,\widetilde{\mathbb{H}}^{(m)} \right]$ and $ \widetilde{\mathbb{F}}_m =\left[\widetilde{\mathbb{F}}^{(1)},...,\widetilde{\mathbb{F}}^{(m)} \right]\in \R^{(m+1)s\times ms}$ are obtained from the BTL algorithm, with 
$$
\widetilde{\mathbb{H}}^{(j)}=\left[\begin{array}{cccc}
	H_{1,j} \\
	\vdots  \\
	H_{j,j} \\
	H_{j+1,j} \\
	\textbf{0}
	\end{array}\right]\, {\rm and} \ \  \widetilde{\mathbb{F}}^{(j)}=\left[\begin{array}{cccc}
	F_{1,j} \\
	\vdots  \\
	F_{j,j} \\
	F_{j+1,j} \\
	\textbf{0}
	\end{array}\right],\quad {\rm for} \ j=0,...,m.
$$
The matrices $H_{i,j}$ and  $F_{i,j}$ constructed in Step 3 of Algorithm \ref{BTL}  are of size $s\times s$ and $\textbf{0}$ is the zero matrix of size $(m-j)s\times s$. We define also the following matrices, 
$$
\widetilde{D}_{m+1}^{(1)}=D_{m+1}^{(1)}\otimes I_s
,\qquad \widetilde{D}_{m+1}^{(2)}= D_{m+1}^{(2)}\otimes I_s
$$
where $D_{m+1}^{(1)}=Diag\{\sigma_1,...,\sigma_{m+1}\}$ and $D_{m+1}^{(2)}=Diag\{\mu_1,...,\mu_{m+1}\}$. With all those notations, we have the following theorem.
\medskip
\begin{theorem}\label{prop1}
Let $\mathbb{V}_{m+1}$ and $\mathbb{W}_{m+1}$ be the bi-orthonormal matrices of $\R^{n\times (m+1)s}$ constructed by 
Algorithm \ref{BTL}. Then we have the following relations 
\begin{equation}\label{eq15}
A\mathbb{V}_{m+1} = \left[\mathbb{V}_{m+1}\mathbb{G}_{m+1}\widetilde{D}_{m+1}^{(1)}-B\mathbb{R}_{m+1}\right]\mathbb{G}_{m+1}^{-1},
\end{equation}
and 
\begin{equation}\label{eq150}
A^T\mathbb{W}_{m+1} = \left[\mathbb{W}_{m+1}\mathbb{Q}_{m+1}\widetilde{D}_{m+1}^{(2)}-C^T\mathbb{L}_{m+1}\right]\mathbb{Q}_{m+1}^{-1}.
\end{equation}
Let $\mathbb{T}_{m+1}$ and $\mathbb{Y}_{m+1}$ be the matrices,
$$ \mathbb{T}_{m+1} = \left[(\sigma_1I-A )^{-1}BR_1,...,(\sigma_{m+1}I-A)^{-1}BR_{m+1} \right] \  {\rm and}$$   $$ \mathbb{Y}_{m+1} = \left[ (\mu_1I-A )^{-T}C^TL_1,...,(\mu_{m+1}I-A)^{-T}C^TL_{m+1} \right],$$ then we have
\begin{equation}\label{eq17}
\mathbb{T}_{m+1}=\mathbb{V}_{m+1}\mathbb{G}_{m+1}\ \ {\rm and } \quad \mathbb{Y}_{m+1}=\mathbb{W}_{m+1}\mathbb{Q}_{m+1},
\end{equation}
where  $ \mathbb{R}_{m+1}=\left[R_1,...,R_{m+1}\right]$ and $ \mathbb{L}_{m+1}=\left[L_1,...,L_{m+1}\right]$. The matrices 
 $\mathbb{G}_{m+1}=\left[\widetilde{\mathbb{H}}^{(0)}\ \  \widetilde{\mathbb{H}}_{m}\right]$ and $\mathbb{Q}_{m+1}=\left[\widetilde{\mathbb{F}}^{(0)}\ \  \widetilde{\mathbb{F}}_{m}\right]$ are block upper triangular matrices of sizes $(m+1)s\times (m+1)s$ and are assumed to be non-singular.
\end{theorem}

\medskip

\begin{proof}
\noindent From Algorithm \ref{BTL}, we have
\begin{equation}\label{eq18}
 V_{j+1}H_{j+1,j}=(\sigma_{j+1}I_n-A)^{-1}BR_{j+1}-\di{\sum_{i=1}^j V_{i}H_{i,j}} \qquad j=1,...,m,
\end{equation}
multiplying (\ref{eq18}) on the left by $(\sigma_{j+1}I_n-A)$ and re-arranging terms,  we get
$$
A \di{\sum_{i=1}^{j+1} V_{i}H_{i,j}}=\sigma_{j+1}\di{\sum_{i=1}^{j+1} V_{i}H_{i,j}}-BR_{j+1} \qquad j=1,...,m,
$$
which gives
$$
A\mathbb{V}_{j+1}\left[\begin{array}{ccc}
	H_{1,j} \\
	\vdots  \\
	H_{j,j} \\
	H_{j+1,j}
	\end{array}\right]=\sigma_{j+1}\mathbb{V}_{j+1}\left[\begin{array}{ccc}
	H_{1,j} \\
	\vdots  \\
	H_{j,j} \\
	H_{j+1,j}
	\end{array}\right]-BR_{j+1},\qquad j=1,\ldots,m,
$$
that  written as
\begin{equation}\label{eq19}
A\mathbb{V}_{m+1}\left[\begin{array}{ccc}
	H_{1,j} \\
	\vdots  \\
	H_{j,j} \\
	H_{j+1,j}\\
	\textbf{0}
	\end{array}\right]=\sigma_{j+1}\mathbb{V}_{j+1}\left[\begin{array}{ccc}
	H_{1,j} \\
	\vdots  \\
	H_{j,j} \\
	H_{j+1,j}\\
	\textbf{0}
	\end{array}\right]-BR_{j+1},\qquad j=1,\ldots,m,
\end{equation}
where $\textbf{0}$ is the zero matrix of size $(m-j)s\times s$. Then for $j=1,\ldots,m$, we have 
\begin{equation}\label{eq20}
A\mathbb{V}_{m+1}\widetilde{\mathbb{H}}^{(j)}=\sigma_{j+1}\mathbb{V}_{j+1}\widetilde{\mathbb{H}}^{(j)}-BR_{j+1},
\end{equation}
now, since $V_1H_{1,0}=(\sigma_{1}I_n-A)^{-1}BR_{1}$, we can deduce from (\ref{eq20}), the following expression
$$
A\mathbb{V}_{m+1}\left[\widetilde{\mathbb{H}}^{(0)},\widetilde{\mathbb{H}}^{(1)},...,\widetilde{\mathbb{H}}^{(m)} \right] = \mathbb{V}_{m+1}\left[\widetilde{\mathbb{H}}^{(0)},\widetilde{\mathbb{H}}^{(1)},...,\widetilde{\mathbb{H}}^{(m)} \right](D_{m+1}^{(1)}\otimes I_s)- B\mathbb{R}_{m+1},$$
which ends the proof of (\ref{eq15}). The same proof can be done for the relation (\ref{eq150}).\\

\noindent For the proof of  \eqref{eq17},  we first use  \eqref{eq18} to obtain 
$$
 \di{\sum_{i=1}^{j+1}V_{i}H_{i,j}}=(\sigma_{j+1}I_n-A)^{-1}BR_{j+1} \qquad j=1,\ldots,m,
$$
which gives
$$
\mathbb{V}_{m+1}\left[\begin{array}{ccc}
	H_{1,j} \\
	\vdots  \\
	H_{j,j} \\
	H_{j+1,j}\\
	\textbf{0}
	\end{array}\right]=(\sigma_{j+1}I_n-A)^{-1}BR_{j+1},\qquad j=1,\ldots,m,
$$
it follows that
$$
\mathbb{V}_{m+1}\left[\widetilde{\mathbb{H}}^{(0)},\widetilde{\mathbb{H}}^{(1)},...,\widetilde{\mathbb{H}}^{(m)} \right] = \left[(\sigma_{1}I_n-A)^{-1}BR_{1},...,(\sigma_{m+1}I_n-A)^{-1}BR_{m+1} \right],$$
which ends the proof  of the first relation of \eqref{eq17}. In the same manner,  we can prove the second relation.
\end{proof}

\medskip

\begin{theorem}\label{theo1}
Let $\sigma,\  \mu \in\C$ be such that $(\omega I-A)$ is invertible for $\omega=\sigma,\ \mu$. Let $\mathbb{V}_m=\left[V_1,...,V_m \right] $ and $\mathbb{W}_m=\left[W_1,...,W_m \right] $  have full-rank, where the $V_i,\ W_i\in \R^{n\times s}$. Let  $R=[r_1,...,r_s],\ L=[l_1,...,l_s] \in \R^{p\times s}$ be a chosen  tangential matrix directions.  Then, 
\begin{enumerate}
\item If $(\sigma I-A)^{-1}Br_i \in Range\{V_1,...,V_m\}$ for $i=1,...,s$, then 
$$H_m(\sigma)R= H(\sigma)R.$$
\item If $(\mu I-A)^{-T}C^Tl_i \in Range\{W_1,...,W_m\}$ for $i=1,...,s$, then 
$$L^TH_m(\sigma)= L^TH(\sigma).$$
\item If both (1) and (2) hold and  in addition we have $\sigma=\mu$, then, $$L^TH_m'(\sigma)R = L^TH'(\sigma)R.$$
\end{enumerate}
\end{theorem}
\vspace{0.2cm}
\begin{proof} \\

\noindent 1) We  follow the same techniques as  those given in \cite{B0041} for the non-block case. Define
$$ \mathcal{P}_m(\omega) = \mathbb{V}_m(\omega I_m - A_m)^{-1}\mathbb{W}_m^{T}(\omega I- A), $$
and
$$ \mathcal{Q}_m(\omega) = (\omega I - A)\mathcal{P}_m(\omega)(\omega I - A)^{-1}=(sI-A)\mathbb{V}_m(\omega I_m-A_m)^{-1}\mathbb{W}_m^{T}. $$
It is easy to verify that $\mathcal{P}_m(\omega)$ and $\mathcal{Q}_m(\omega)$ are projectors. Moreover, for all $\omega$ in a neighborhood of $\sigma$ we have  $$\mathcal{V}_m =Range\{V_1,...,V_m\}=Range(\mathcal{P}_m(\omega)) = Range(I -\mathcal{P}_m(\omega)),$$
and 
$$\mathcal{W}_m^{\bot} =Range\{W_1,...,W_m\}^{\bot}=Ker(\mathcal{Q}_m(\omega)) = Ker(I -\mathcal{Q}_m(\omega)).$$
 Observe that
\begin{equation}\label{eq13}
H(\omega) - H_m(\omega) = C(\omega I - A)^{-1}(I -\mathcal{Q}_m(\omega))(\omega I -A)(I -\mathcal{P}_m(\omega ))(\omega I-A)^{-1}B.
\end{equation}
Evaluating the expression \eqref{eq13} at $\omega = \sigma$ and multiplying by $r_i$ from the right, yields the first assertion, and evaluating the same expression  at $\omega = \mu$ and multiplying by $l_i^T$ from the left, yields the second assertion \\

\noindent 2) Now if both (1) and (2) hold and   $\sigma=\mu$,  notice that 
$$((\sigma+\varepsilon)I-  A)^{-1} = (\sigma I-A)^{-1}-\varepsilon(\sigma I- A)^{-2}+O(\varepsilon^2),$$
and
$$((\sigma+\varepsilon)I_m - A_m)^{-1} = (\sigma I_m-A_m)^{-1}-\varepsilon(\sigma I_m- A_m)^{-2}+O(\varepsilon^2).$$
Therefore, evaluating (\ref{eq13}) at $s = \sigma + \varepsilon$, multiplying by $l_j^T$ and $r_i$,  from the left and the right respectively, for $i,j=1,...,s$,  we get
$$ l_j^T H(\sigma + \varepsilon)r_i -l_j^TH_m(\sigma + \varepsilon)r_i = O(\varepsilon^2).$$
Now notice that since $l_j^T H(\sigma)r_i = l_j^T H_m(\sigma)r_i$, we have 
$$\displaystyle \lim_{\varepsilon\longrightarrow 0}\, \left [ \frac{1}{\varepsilon}( l_j^T H(\sigma+\varepsilon)r_i -l_j^T H(\sigma)r_i)-\frac{1}{\varepsilon}(l_j^T H_m(\sigma + \varepsilon)r_i -l_j^TH_m(\sigma)r_i)\right ] = 0,$$
 which proves the third assertion. 
\end{proof}\\

\noindent In the following theorem, we give the exact expression of the residual norms in a simplified and economical  computational form. 

 
\section{An adaptive choice of the interpolation points and tangent directions}\quad 

 \noindent In the section, we will see how to chose  the interpolation points $\{\sigma_i\}_{i=1}^m$, $\{\mu_i\}_{i=1}^m$ and tangential  directions $\{R_i\}_{i=1}^m$, $\{L_i\}_{i=1}^m$,  where  $R_i,\ L_i\in \R^{p\times s}$. In this paper we adopted the adaptive approach, inspired  by the work in \cite{B008}. For this approach, we  extend our subspaces $\mathcal{K}_m(A,B)$ and 
 $\widetilde{\mathcal{K}}_m(A^T,C^T)$ 
by adding  new blocks $\widetilde{V}_{m+1}$ and $\widetilde{W}_{m+1}$ defined as follows
\begin{equation}\label{eq24}
\widetilde{V}_{m+1}=(\sigma_{m+1}I_n-A)^{-1}BR_{m+1}\, {\rm and}\, 
\widetilde{W}_{m+1}=(\sigma_{m+1}I_n-A)^{-T}C^TL_{m+1},
\end{equation}
where the  new interpolation point $\sigma_{m+1}$, $\mu_{m+1}$ and the new tangent direction $R_{m+1}$, $L_{m+1}$
are computed  as follows
\begin{equation}\label{eq25}
(R_{m+1},\sigma_{m+1}) = arg\di\max_{\tiny{\left.\begin{array}{ccc}
\omega \in S_m \\
R\in \R^{p\times s}\\
\|R\|_2=1
 \end{array}\right.}}
 \left\|R_B(\omega) R\right\|{_2},
\end{equation}
\begin{equation}\label{eq250}
(L_{m+1},\mu_{m+1}) = arg\di\max_{\tiny{\left.\begin{array}{ccc}
\omega \in S_m \\
L\in \R^{p\times s}\\
\|L\|_2=1
 \end{array}\right.}}
 \left\|R_C(\omega) L\right\|{_2}.
\end{equation}
Here $S_m \subset \C^{+}$ is defined  as the convex hull of $\{-\lambda_1,...,-\lambda_m\}$ where $\{\lambda_i\}_{i=1}^m$ are  the eigenvalues of the matrix $A_m$. For solving  the problem \eqref{eq25}, we proceed  as follows. First we compute the next interpolation point,  by computing the norm of $R_B(\omega)$ for each $\omega$ in $S_m$, i.e we solve the following problem,
 \begin{equation}\label{eq26}
\left.\begin{array}{lll}
\sigma_{m+1} & = arg \di{\max_{\omega\in  S_m}}\|R_B(\omega)\|{_2}.
\end{array}\right.
\end{equation}
 Then the tangent direction $R_{m+1}$ is computed  by evaluating \eqref{eq25} at $\omega=\sigma_{m+1}$,
\begin{equation}\label{eq27}
R_{m+1} = arg \di\max_{\tiny{\left.\begin{array}{ccc}
R\in \R^{p\times s}\\
\|R\|_2=1
 \end{array}\right.}}\left\|R_B(\sigma_{m+1})R\right\|_2.
\end{equation}
\noindent We can easily prove that the  tangent matrix direction $R_{m+1}$ is given as  $$R_{m+1}=[r_{1}^{(m+1)},...,r_{s}^{(m+1)}],$$
 where the $r_{i}^{(m+1)}$'s are the right singular vectors corresponding to the $s$ largest singular values of the matrix $R_B(\sigma_{m+1})$. This approach of maximizing the residual norm, works efficiently for small to medium matrices, but cannot be used  for large scale systems. To overcome this problem, we give the following proposition.
 \medskip
\begin{proposition}\label{prop2}
Let $R_B(\omega)=B-(\omega I_n -A)\mathbb{V}_m U_m^B(\omega)$ and $R_C(\omega)=C^T-(\omega I_n -A)^T\mathbb{W}_m U_m^C(\omega)$ be the residuals  given in \eqref{eq12} and \eqref{eq13}, where $U_m^B(\omega)=(\omega I-A_m)^{-1}\mathbb{W}_m^TB$ and $U_m^C(\omega)=(\omega I-A_m)^{-T}\mathbb{V}_m^TC^T$. 
Then we have  the following  new expressions 
\begin{equation}
\label{eq21}
R_B(\omega) = (I_n-\mathbb{V}_m\mathbb{W}_m^T)B\left[I_p-\mathbb{ R}_m\mathbb{G}_m^{-1}U_m^B(\omega)\right],
\end{equation}
and
\begin{equation}
\label{eq210}
R_C(\omega) = ( I_n-\mathbb{W}_m\mathbb{V}_m^T)C^T\left[I_p-\mathbb{ L}_m\mathbb{Q}_m^{-1}U_m^C(\omega)\right].
\end{equation}
\end{proposition}
\begin{proof}\\
\noindent The residual $R_B(\omega)$ can be written as 
$$\left.\begin{array}{llll}
R_B(\omega) &=B- \omega\mathbb{V}_m U_m^B(\omega) + A\mathbb{V}_mU_m^B(\omega) \\[0.2cm]
        &= B+A\mathbb{V}_m U_m^B(\omega) - \mathbb{V}_m(\omega I_{ms}-A_m )(\omega I_{ms}-A_m)^{-1}\mathbb{W}_m ^{T}B \\[0.2cm]
        &\quad \quad  - \mathbb{V}_m A_m(\omega I_{ms}-A_m)^{-1}\mathbb{W}_m ^{T}B\\[0.2cm]
        &=B+A\mathbb{V}_m U_m^B(\omega) -\mathbb{V}_m\mathbb{W}_m^{T}B-\mathbb{V}_m A_m U_m^B(\omega) \\[0.2cm]
        &=(I_n-\mathbb{V}_m\mathbb{W}_m^{T})B+(A\mathbb{V}_m-\mathbb{V}_m A_m)U_m^B(\omega),
\end{array}\right.$$
using Equation \eqref{eq15}, we get 
$$A_m=\mathbb{W}_{m}^TA\mathbb{V}_{m} = \left[\mathbb{G}_{m}\widetilde{D}_{m}^{(1)}-B_m\mathbb{R}_{m}\right]\mathbb{G}_{m}^{-1},$$
which gives,
$$
A\mathbb{V}_m-\mathbb{V}_m A_m=\left[I-\mathbb{V}_m\mathbb{W}_m^{T}\right]B\mathbb{R}_m\mathbb{G}_m^{-1},
$$
which proves \eqref{eq21}. In the same way we can  prove \eqref{eq210}.
\end{proof}
 \medskip
 
\noindent The expression of $R_B(\omega)$ given in  \eqref{eq21}  allows us to significantly reduce the computational cost
while seeking the next pole and direction. In fact, applying the skinny QR decomposition $$(I_n-\mathbb{V}_m\mathbb{W}_m^T)B=QL,$$ we get the simplified residual norm 
\begin{equation}\label{eq29}
\left\|R_B(\omega)\right\|_2=\left\|L[I_p-\mathbb{ R}_m\mathbb{G}_m^{-1}U_m^B(\omega)]\right \|_2.
\end{equation}
This  means that, solving the problem \eqref{eq25} requires  only the computation  of matrices of size
$ms \times ms$ for each value of $\omega$. \\

\noindent The next algorithm, summarizes all the steps of the adaptive choice of tangent interpolation points and tangent directions. 
\begin{algorithm}
\caption{The Adaptive Block  Tangential  Lanczos (ABTL) algorithm}\label{ABTL}
\begin{itemize}
\item Given $A$, $B$, $C$, $m_{max}$. 
\item Outputs: $\mathbb{V}_{m_{max}}$, $\mathbb{W}_{m_{max}}$.
\item Compute $(\sigma_1 I_n-A)^{-1}BR_1 = H_{1,0}V_1$ and $(\mu_{1}I_n-A)^{-T}C^TL_{1} = F_{1,0}W_1$ such that $W_1^TV_1=I_p$.
\item Initialize: $\mathbb{V}_1= [V_1 ]$, $\mathbb{W}_1= [W_1]$.
\begin{enumerate}
\item For $m=1:m_{max}$   
\item Set $A_m = \mathbb{W}_m^{T}A\mathbb{V}_m$, $B_m = \mathbb{W}_m^{T}B$, $C_m = C\mathbb{V}_m$.
\item Compute $\sigma_{m+1}$, and $\mu_{m+1}$
\begin{itemize}
\item  Compute $\{\lambda_1,...,\lambda_m\}$ eigenvalues of $A_m$.
\item Determine $ S_m$, convex hull of $\{-\lambda_1,...,-\lambda_m\}$.
\item Solve (\ref{eq26}). The same for $\mu_{m+1}$.
\end{itemize}
 \item Compute right and left vectors $R_{m}$, $L_{m}$.
\item   $\widetilde{V}_m= (\sigma_{m+1}I_n-A)^{-1}BR_{m+1}$,\quad $\widetilde{W}_{m+1} = (\mu_{m+1}I_n-A)^{-T}C^TL_{m+1}$.
\item   For i = 1,...,m
\begin{itemize}
\item  $H_{i,m} = W_i^T\widetilde{V}_{m+1}$,\qquad \qquad  \textbf{--} $F_{i,m} = V_i^T\widetilde{W}_{m+1}$,
\item  $\widetilde{V}_{m+1} = \widetilde{V}_{m+1} - V_i H_{i,m}$, \quad \;\textbf{--} $\widetilde{W}_{m+1} = \widetilde{W}_{m+1} - W_i F_{i,m}$,
\end{itemize}
\item End.
\item $\widetilde{V}_{m+1}=V_{m+1}H_{m+1,m}$,\ $\widetilde{W}_{m+1}=W_{m+1}F_{m+1,m}$. \qquad (QR decomposition).
\item $W_{m+1}^TV_{m+1}=P_{m+1}D_{m+1}Q_{m+1}^T$. \qquad (Singular Value Decomposition).
\item $V_{m+1}=V_{m+1}Q_{m+1}D_{m+1}^{-1/2}$, \  $W_{m+1}=W_{m+1}P_{m+1}D_{m+1}^{-1/2}$.
\item $H_{m+1,m}=D_{m+1}^{1/2}Q_{m+1}^TH_{m+1,m}$, \ $F_{m+1,m}=D_{m+1}^{1/2}P_{m+1}^TF_{m+1,m}$.
\item  $\mathbb{V}_{m+1}=\left[\mathbb{V}_m , \; V_{m+1}\right]$,\qquad $\mathbb{W}_{m+1}=\left[\mathbb{W}_m , \; W_{m+1}\right]$.
\item End.
\end{enumerate}
\end{itemize}
\end{algorithm}

\section{Model Reduction of Second-Order Systems}\quad 

\noindent Linear PDEs modeling structures in many areas of engineering (plates, shells, beams ...) are often second order in time  see
for example \cite{B23,B24, B25}. The spatial semi-discretization of its models by a method
of finite elements leads to systems that write in the form:
\begin{equation}\label{eq6.01}
\left\{\begin{array}{lllll}
M\ddot{q}(t)+D\dot{q}(t) +Kq(t) & = & Bu(t)\\
  y(t)  & = & Cq(t),\\
 \end{array}\right.
\end{equation}
where $M\in \mathbb{R}^{n\times n} $ is the mass matrix, $D\in \mathbb{R}^{n\times n} $ is the damping matrix and $K\in \mathbb{R}^{n\times n} $ the stiffness matrix. When
the source term $Bu(t)$ is null, the system is said to be free, otherwise, it is said forced. If $D = 0$, the system  is said to be undamped. We assume that the mass matrix $M$ is invertible, then the system \eqref{eq6.01} can be written as 
\begin{equation}\label{eq6.02}
\left\{\begin{array}{lllll}
\ddot{q}(t)+D_M\dot{q}(t) +K_Mq(t) & = & B_Mu(t)\\
  y(t)  & = & Cq(t),\\
 \end{array}\right.
\end{equation}
where $D_M=M^{-1}D$, $K_M=M^{-1}K$ and $B_M=M^{-1}B$, for simplicity we still denote $K,\  D,$  and $B $ instead of $D_M$, $K_M$ and $B_M$. The transfer function associated with the system \eqref{eq6.02} is given by using the Laplace transform as:
\begin{equation}\label{eq6.03}
F(\omega) := C(\omega^2 I_n+\omega D+K)^{-1}B \in \R^{p\times p}.
\end{equation}
Usually, it's difficult to have the efficient solution of various control or simulation tasks because the original system is too large to allow it. In order to solve this problem, methods
that produce a reduced system of size $m \ll n$ that preserves the essential
properties of the full order model have been developed. The reduced model have the following form:
\begin{equation}\label{eq6.04}
\left\{\begin{array}{lllll}
\ddot{q}_m(t)+D_m\dot{q}_m(t) +K_mq_m(t) & = & B_mu(t)\\
  y_m(t)  & = & C_mq(t),\\
 \end{array}\right.
\end{equation}
where $D_m, \ K_m\in \mathbb{R}^{m\times m} $, $B_m,\ C_m^T\in \mathbb{R}^{m\times p} $  and $q_m(t)\in \mathbb{R}^{m}$. The transfer function associated to the system \eqref{eq6.04} is given by:
\begin{equation}\label{eq6.05}
F_m(\omega) := C_m(\omega^2 I_m+\omega D_m+K_m)^{-1}B_m \in \R^{p\times p}.
\end{equation}
Second-order systems \eqref{eq6.02} can be written as  a first order linear systems. In fact,
\begin{equation}\label{eq6.06}
\left\{\begin{array}{cclll}
\left[\begin{array}{ccc}
\dot{q}(t)\\
\ddot{q}(t)\\
 \end{array}\right] & = &\left[\begin{array}{cccc}
0 &  I_n\\
-K  & -D\\
 \end{array}\right]\left[\begin{array}{ccc}
q(t)\\
\dot{q}(t)\\
 \end{array}\right] +\left[\begin{array}{ccc}
0\\
B\\
 \end{array}\right]u(t)\\[0.3cm]
   y(t)  &  = & \left[\begin{array}{ccc}
C & 0\\
 \end{array}\right]\left[\begin{array}{ccc}
q(t)\\
\dot{q}(t)\\
 \end{array}\right], \end{array}\right.
\end{equation}
which is equivalent to
\begin{equation}\label{eq6.07}
 \left\{\begin{array}{lllll}
\dot{x}(t)  & = & \mathcal{A}x(t) +  \mathcal{B}u(t)\\
  y(t)  & = & \mathcal{C}x(t),\\
 \end{array}\right.
\end{equation}
with $x(t)=\left[\begin{array}{ccc}
q(t)\\
\dot{q}(t)\\
 \end{array}\right]$, $\mathcal{A}=\left[\begin{array}{ccc}
0  & I_n\\
-K & -D\\
 \end{array}\right]$,  $\mathcal{B}=\left[\begin{array}{ccc}
0 \\
B\\
 \end{array}\right]$ and $\mathcal{C}=\left[\begin{array}{ccc}
C  & 0\\
 \end{array}\right]$.\\[0.2cm]
Thus, the corresponding transfer function is defined as,
 \begin{equation}\label{eq6.08}
\mathcal{F}(\omega) := \mathcal{C}(\omega I_{2n}-\mathcal{A})^{-1}\mathcal{B} \in \R^{p\times p}.
\end{equation}
We note that $\mathcal{F}(\omega)=F(\omega)$. In fact, setting
$$ X= (\omega I_{2n}-\mathcal{A})^{-1}\mathcal{B}=\left[\begin{array}{ccc}
X_1 \\
X_2\\
 \end{array}\right],$$
 wich gives $ \mathcal{F}(\omega)=\mathcal{C}X$, where $X$ verifies $(\omega I_{2n}-\mathcal{A})X=\mathcal{B}$. Using the expressions of the matrices $\mathcal{A}$, $\mathcal{B}$ and $\mathcal{C}$, we get,
 $$
 (\omega^2 I_n+\omega D+K)X_1=B\ \text{and}\ \mathcal{F}(\omega)=CX_1.  
 $$
 Hence
 $$
 \mathcal{F}(\omega)=F(\omega)=C(\omega^2 I_n+\omega D+K)^{-1}B.
 $$
 We can reduce the second-order system \eqref{eq6.02} by
applying  linear model reduction technique presented in the previous section, to $(A,\ B,\ C)$  to
yield a small linear system $(A_m,\ B_m,\ C_m)$. Unfortunately, there is no guarantee
that the matrices defining the reduced system have the necessary structure
 to preserve the second-order form of the original system. For that we follow the model reduction techniques of second-order structure-preserving, presented in \cite{B26,B27,B28}.
 \subsection{The structure-preserving of the  second-order reduced model }\quad 

 \noindent Using the Krylov subspace-based methods discussed in the previous  section do not guaranty the
second-order structure when applied to the linear system \eqref{eq6.07}. the authors in \cite{B26,B28} proposed a result, that gives a simple sufficient condition to satisfy the interpolation condition and produce a second order reduced system.\\

\begin{lem}
Let $(\mathcal{A},\ \mathcal{B},\ \mathcal{C})$ be the state space realization defined in \eqref{eq6.07}. If we project the
state space realization with $2n \times 2ms$ bloc diagonal matrices
 $$\mathcal{V}_m=\left[\begin{array}{ccc}
\mathcal{V}_m^1  & 0\\
0 & \mathcal{V}_m^2\\
 \end{array}\right],\ \mathcal{W}_m=\left[\begin{array}{ccc}
\mathcal{W}_m^1  & 0\\
0 & \mathcal{W}_m^2\\
 \end{array}\right],\ \mathcal{W}_m^T\mathcal{V}_m=I_{2ms},$$
where $\mathcal{V}_m^1,\ \mathcal{V}_m^2,\ \mathcal{W}_m^1 \text{and} \ \mathcal{W}_m^2\in \R^{n\times ms}$, then the reduced transfer function
$$\mathcal{F}_m(\omega) := \mathcal{C}\mathcal{V}_m(\omega I_{2mp}-\mathcal{W}_m^T\mathcal{A}\mathcal{V}_m)^{-1}\mathcal{W}_m^T\mathcal{B},$$
is a second order transfer function, on condition that  the matrix $(\mathcal{W}_m^2)^T\mathcal{V}_m^1$ is invertible.
\end{lem}
\begin{theorem}
Let $\mathcal{F}(\omega) := \mathcal{C}(\omega I_{2n}-\mathcal{A})^{-1}\mathcal{B}= C(\omega^2 I_n+\omega D+K)^{-1}B$, with
$$\mathcal{A}=\left[\begin{array}{ccc}
0  & I_n\\
-K & -D\\
 \end{array}\right],\ \mathcal{B}=\left[\begin{array}{ccc}
0 \\
B\\
 \end{array}\right]\ and \ \mathcal{C}=\left[\begin{array}{ccc}
C  & 0\\
 \end{array}\right],$$
 be a second order transfer function. Let $\mathbb{V}_m,\ \mathbb{W}_m \in \R^{2n\times ms}$ be defined as:
 $$\mathbb{V}_m=\left[\begin{array}{ccc}
\mathbb{V}^1_m\\
\mathbb{V}^2_m\\
 \end{array}\right],\quad \mathbb{W}_m=\left[\begin{array}{ccc}
\mathbb{W}^1_m\\
\mathbb{W}^2_m\\
 \end{array}\right],$$
 where $\mathbb{V}_m^1,\ \mathbb{V}_m^2,\ \mathbb{W}_m^1 \ \text{and} \ \mathbb{W}_m^2\in \R^{n\times ms}$, with $(\mathbb{W}_m^1)^T\mathbb{V}_m^1=(\mathbb{W}_m^2)^T\mathbb{V}_m^2=I_{ms}$. Let us construct the $2n \times 2ms$ projecting matrices as
 $$\mathcal{V}_m=\left[\begin{array}{ccc}
\mathbb{V}^1_m & 0\\
0 &\mathbb{V}^2_m\\
 \end{array}\right],\quad \mathcal{W}_m=\left[\begin{array}{ccc}
\mathbb{W}^1_m & 0\\
0 & \mathbb{W}^2_m\\
 \end{array}\right].$$
 Define the second order transfer function of order $m$ by
 $$\left.\begin{array}{llll}
 \mathcal{F}_m(\omega) &=& \mathcal{C}\mathcal{V}_m(\omega I_{2mp}-\mathcal{W}_m^T\mathcal{A}\mathcal{V}_m)^{-1}\mathcal{W}_m^T\mathcal{B}\\[0.2cm]
  &=& \mathcal{C}_m(\omega I_{2mp}-\mathcal{A}_m)^{-1}\mathcal{B}_m,
 \end{array}\right.$$
 if we have 
 $$ Span\{(\sigma_1 I-\mathcal{A})^{-1}\mathcal{B}R_1,...,(\sigma_m I-\mathcal{A})^{-1}\mathcal{B}R_m \} \subseteq Range(\mathbb{V}_m),$$
 and 
  $$ Span\{(\mu_1 I-\mathcal{A})^{-T}\mathcal{C}^TL_1,...,(\mu_m I-\mathcal{A})^{-T}\mathcal{C}^TL_m \} \subseteq Range(\mathbb{W}_m),$$
where $\sigma_i,\ \mu_i$, are the interpolation points, and $R_i,\ L_i\in \R^{p\times s},$ for $ i=1,...m$ are the tangential directions. Then the reduced order transfer function $\mathcal{F}_m(\omega)$     interpolates the values of the original transfer function $\mathcal{F}(\omega)$  and preserves the structures of the second-order model provided that the matrix $(\mathbb{W}_m^2)^T\mathbb{V}_m^1$  is non-singular.

\end{theorem}

\section{Numerical experiments}\quad 

\noindent In this section, we present some numerical examples to show the effectiveness of the  Adaptive Block Tangential Lanczos (ABTL) algorithm. All the experiments were carried out using the CALCULCO computing platform, supported by SCoSI/ULCO (Service Commun du Syst\`eme d'Information de l'Unive-rsit\'e du Littoral C\^ote d'Opale). The algorithms were coded in Matlab R2018a.  We used the  following functions  from LYAPACK \cite{B011}:
\begin{itemize}
\item  lp$\_$lgfrq: Generates a set of logarithmically distributed frequency sampling points.
\item  lp$\_$gnorm: Computes $\|H(j\omega)-H_{m}(j\omega)\|_2$.
\end{itemize}
\noindent We used various matrices from LYAPACK and from the Oberwolfach collection\footnote{Oberwolfach model reduction benchmark collection 2003. http://www.imtek.de/simulation/benchmark}. These matrix tests are reported
in Table \ref{tab1} with   different values of $p$ and the used values of $s$.

\begin{table}[h!]
\begin{center}
\caption{Matrix Tests}\label{tab1}
   \begin{tabular}{ llllllllllllllllllllllllllllll }
     \hline
     Model & & & & & & & & & &  \;\; $n$ && \;\;$p$ & & \;\;$s$ \\
\hline 
FDM10000                  & & & & & & & & & &    n = 40 000 &&   p = 6   &&  s = 3\\ 
FDM90000                  & & & & & & & & & &    n = 90 000 &&   p = 6   &&  s = 3\\ 
 1DBeam-LF1000    & & & & & & & & & &    n = 19 998  &&   p = 4    &&  s = 2\\    
  1DBeam-LF5000    & & & & & & & & & &    n = 19 994  &&   p = 4    &&  s = 2\\  
   RAIL79841   & & & & & & & & & &    n = 17 361  &&   p = 12    &&  s = 2\\  
     \hline
   \end{tabular}
\end{center}
\end{table}

\subsection{Example 1: FDM model}\label{fdm}
\noindent The finite differences semi-discretized heat equation will serve as the most basic test example here.
Its corresponding matrix $A$, is obtained from the centered finite difference discretization of the operator,
$$ L_A(u) = \Delta u-f(x,y)\frac{\partial u}{\partial x}-g(x,y)\frac{\partial u}{\partial y}-h(x,y)u,$$
on the unit square $[0, 1]\times [0, 1]$ with homogeneous Dirichlet boundary conditions with
$$
\left\{\begin{array}{lll}
f(x,y) & = & log(x + 2y+1) \\
g(x,y) & = & e^{x + y} \\
h(x,y) & = & x + y. \\
\end{array}\right.
$$
The matrices $B$ and $C$ were random matrices with entries uniformly distributed
in $[0, 1]$. The
dimension of $A$ is $n = n_0^2$, where $n_0^2$ is the number of inner grid points in each direction. The advantages of this model are: \begin{itemize}
\item It's easy to understand.

\item The discretization using the finite difference method (FDM) is easy to implement.

\item It allows for simple generation of almost arbitrary size test problems.
\end{itemize}

\noindent In Table \ref{tab2}, we compared the execution times and the ${\mathcal{H_{\infty}}}$  norm $\parallel H-H_m \parallel_{ {\mathcal H}_{\infty}}$ of the ABTL algorithm  with the Iterative Rational Krylov Algorithm (IRKA \cite{B37}) and the adaptive tangential method represented by Druskin and Simonsini (TRKSM) see for more details \cite{B010},  with different values of  $m$.  We notice that  the obtained timing didn't contain the execution times used to obtain the errors. As can be seen from   the results in  Table \ref{tab2},  the cost of IRKA method is much higher than the cost required with  the adaptive block tangential Lanczos method. 

\begin{table}[h!]
	\begin{center}
		\caption{The calculation time and the ${\mathcal{H_{\infty}}}$  error-norm}\label{tab2}
		\begin{small}
		\begin{tabular}{  ll|lllllllllllll }
			\hline
			Model      & &   \qquad ABTL    &  \qquad  IRKA  & \qquad  TRKSM  \\
				\hline 	
			& &      Time \quad  Err-$\mathcal{H}_{\infty}$  &    Time \quad Err-$\mathcal{H}_{\infty}$   &    Time \quad Err-$\mathcal{H}_{\infty}$   \\
			\hline 
			 FDM40.000   & m=20 & $13.29s \quad   5.39\times 10^{-4} $  &  $126.28s \quad  1.06\times 10^{-4}$ &  $34.89s \quad  7.94\times 10^{-4}$ \\ 
	& m=30 & $20.79s \quad   7.87\times 10^{-5} $  &  $188.91s \quad  2.24\times 10^{-5}$ &  $36.82s \quad  6.42\times 10^{-5}$ \\
  & m=40 & $27.21s \quad   1.48\times 10^{-5} $  &  $269.36s \quad  1.56\times 10^{-5}$ &  $38.12s \quad  1.93\times 10^{-5}$ \\ 
			\hline 
FDM90.000  & m=20 & $43.29 \quad  1.27\times 10^{-2} $  &  $>2000s$ \quad  -- --&  $126.97s \quad  2.61\times 10^{-1}$ \\  
& m=30 & $64.11 \quad   1.36\times 10^{-3} $  &  $>2000s$ \quad  -- --&  $179.55s \quad  1.43\times 10^{-1}$ \\ 
			& m=40 & $87.19 \quad   6.25\times 10^{-4} $  &  $>2000s$ \quad  -- --&  $186.38s \quad  4.61\times 10^{-2}$ \\ 
			\hline 
		\end{tabular}
		  \end{small}
	\end{center}
\end{table}
\subsection{Example 2: Linear 1D Beam}\quad 

\noindent  Moving structures are an essential part for many micro-system devices, among them fluidic components like pumps and electrically controllable valves, sensing cantilevers, and optical structures. While the single component can easily be simulated on a usual desktop computer, the calculation of a system of many coupled devices still presents a challenge. This challenge is raised by the fact that many of these devices show a nonlinear behavior. This model describes a slender beam with four degrees of freedom per node: "$x$ the axial displacement", "$\Theta_x$ the axial rotation", "$y$ the flexural displacement" and "$\Theta_z$ the flexural rotation". The model is from the Oberwolfach collection.  The matrices are obtained by using  the finite element discretization presented in \cite{B37}. We used two examples of linear 1D Beam model:
\begin{table}[h!]
\begin{center}
   \begin{tabular}{ llllllllllllllllllllllllllllll }
     \hline
     The file name & & Degrees of freedom  & Num. nodes &   Dimension $n$  \\
\hline 
 1DBeam-LF100    & &  flexural ($\Theta_z$ and $y$)& 10000 &    n = 19998  \\  
 1DBeam-LF5000      & & ( $\Theta_z$ and $y$ ), ( $\Theta_x$ and $x$ ) & 50000 &    n = 19994  \\ 
     \hline
   \end{tabular}
\end{center}
\end{table}

\begin{figure}[!h]
   \begin{minipage}[c]{.46\linewidth}
      \includegraphics[width=1.0\textwidth]{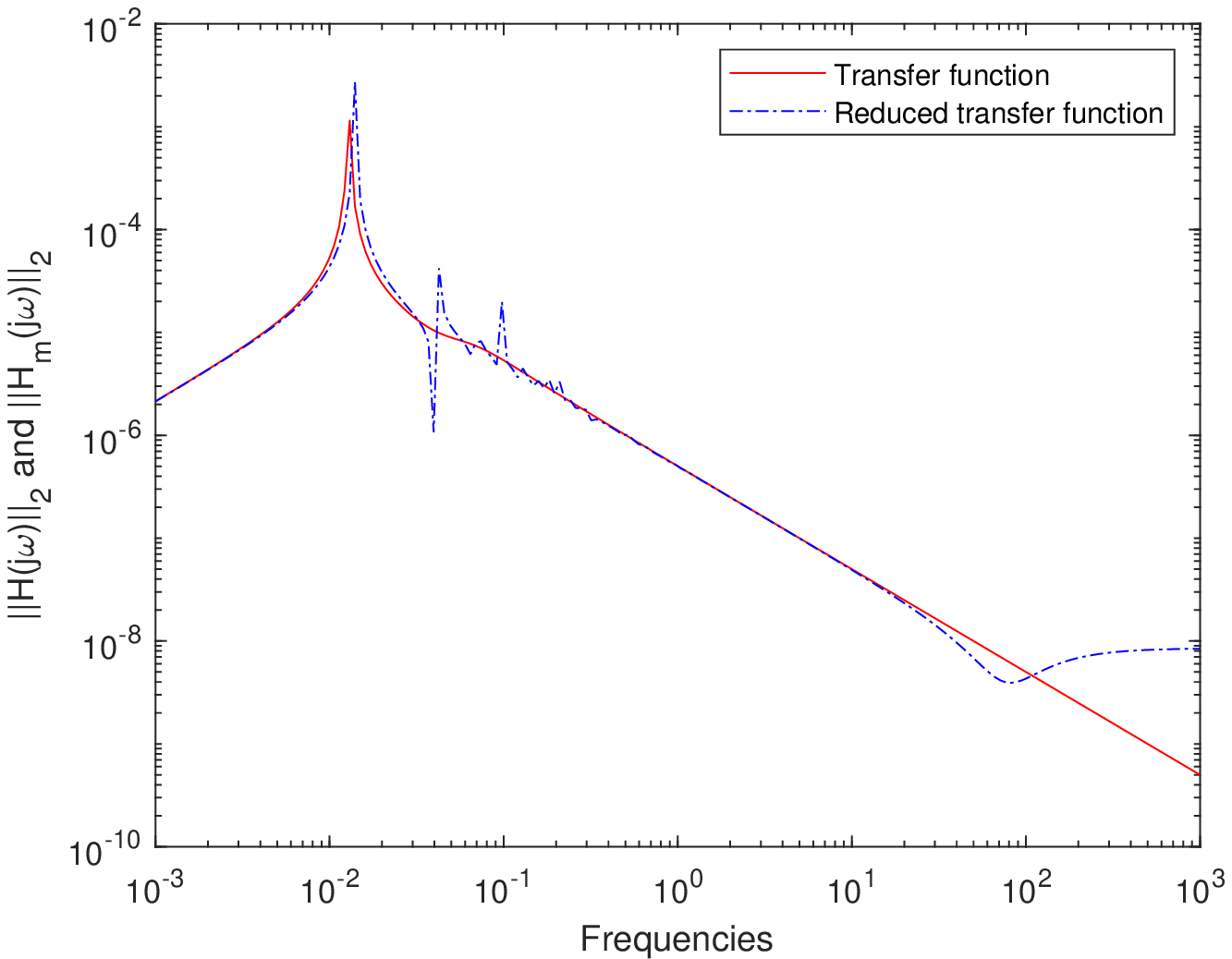}
     \caption{1DBeam-LF5000 model: m=20.}
      \label{fig01}   
   \end{minipage} \hfill
   \begin{minipage}[c]{.46\linewidth}
      \includegraphics[width=1.0\textwidth]{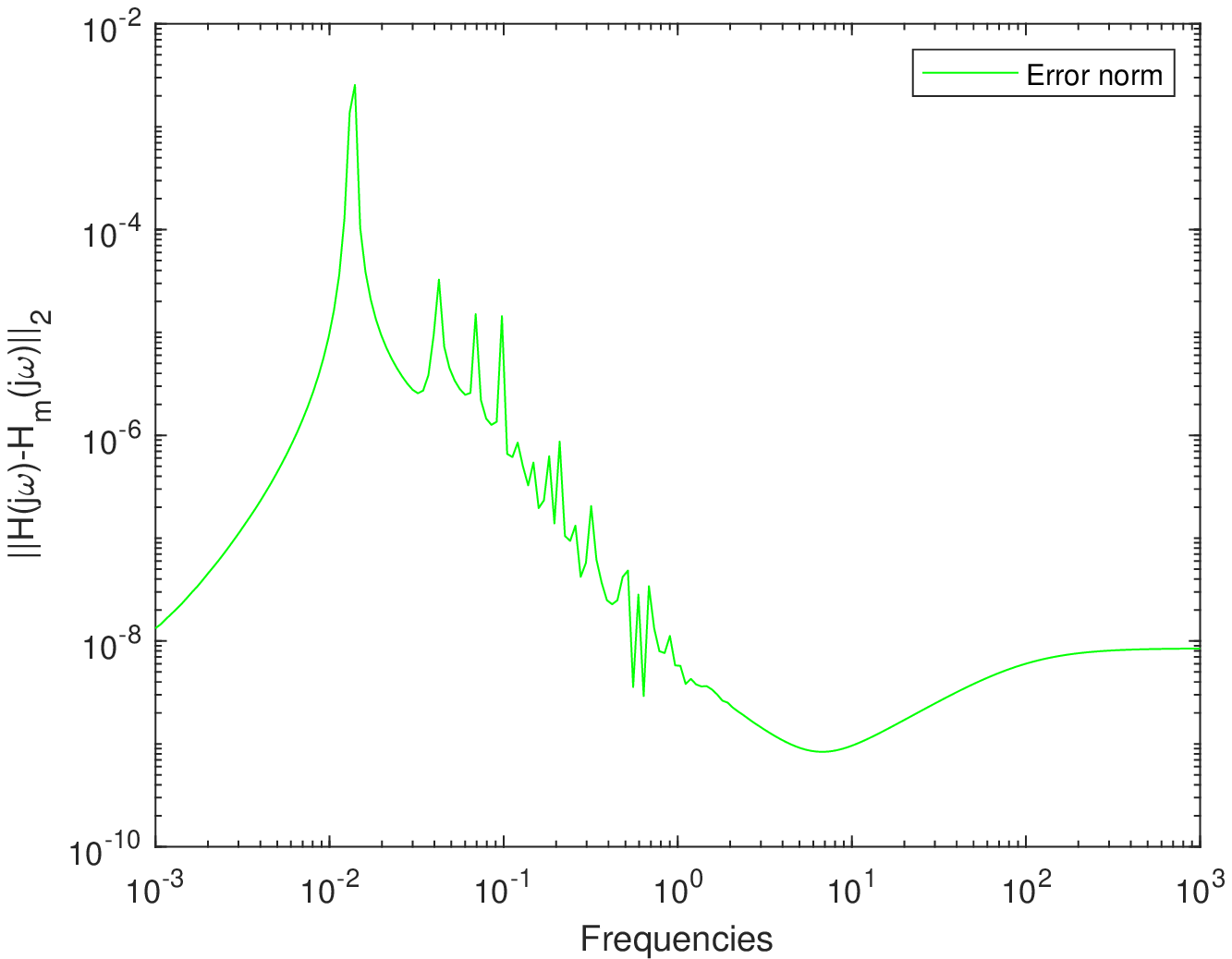}
      \caption{The error norm.}
     \label{fig02}
   \end{minipage}
\end{figure}
\noindent The Figures  \ref{fig01}  above represent the  norm of the original transfer function  $\|H(j\omega)\|_2$ and the  norm of the reduced transfer function  $\|H_m(j\omega)\|_2$ versus the frequencies  $\omega \in [10^{-3},\; 10^3]$ of the  1Dbeam-LF50000 model and it is a second-order
model of dimension $2\times n = 39988$ with one input and one output. The Figure \ref{fig02}  represents the exact error $\|H(j\omega) - H_m(j\omega)\|_2$ versus the frequencies.

\begin{figure}[!h]
   \begin{minipage}[c]{.46\linewidth}
      \includegraphics[width=1.0\textwidth]{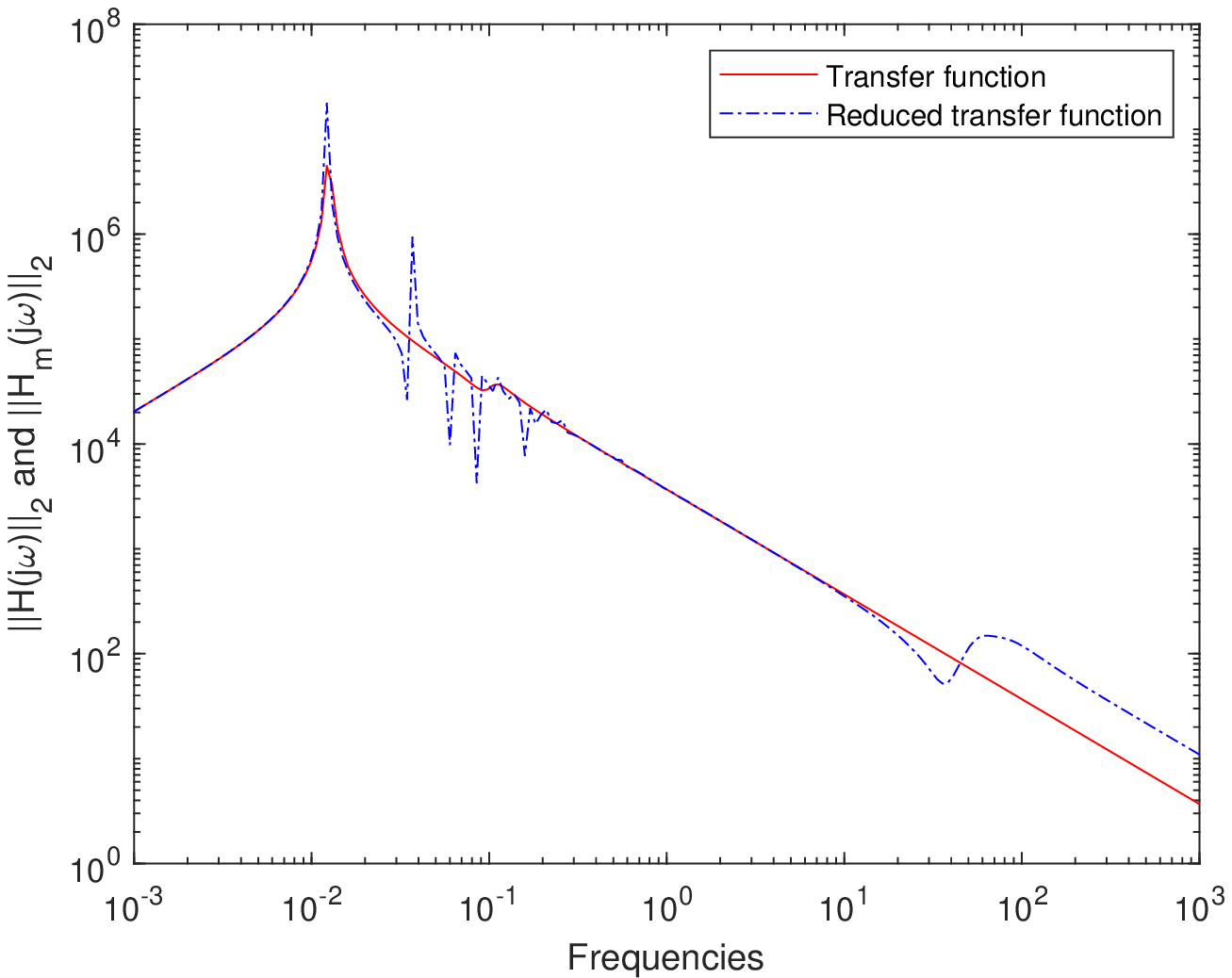}
     \caption{1DBeam-LF5000 model, m=20.}
      \label{fig03}   
   \end{minipage} \hfill
   \begin{minipage}[c]{.46\linewidth}
      \includegraphics[width=1.0\textwidth]{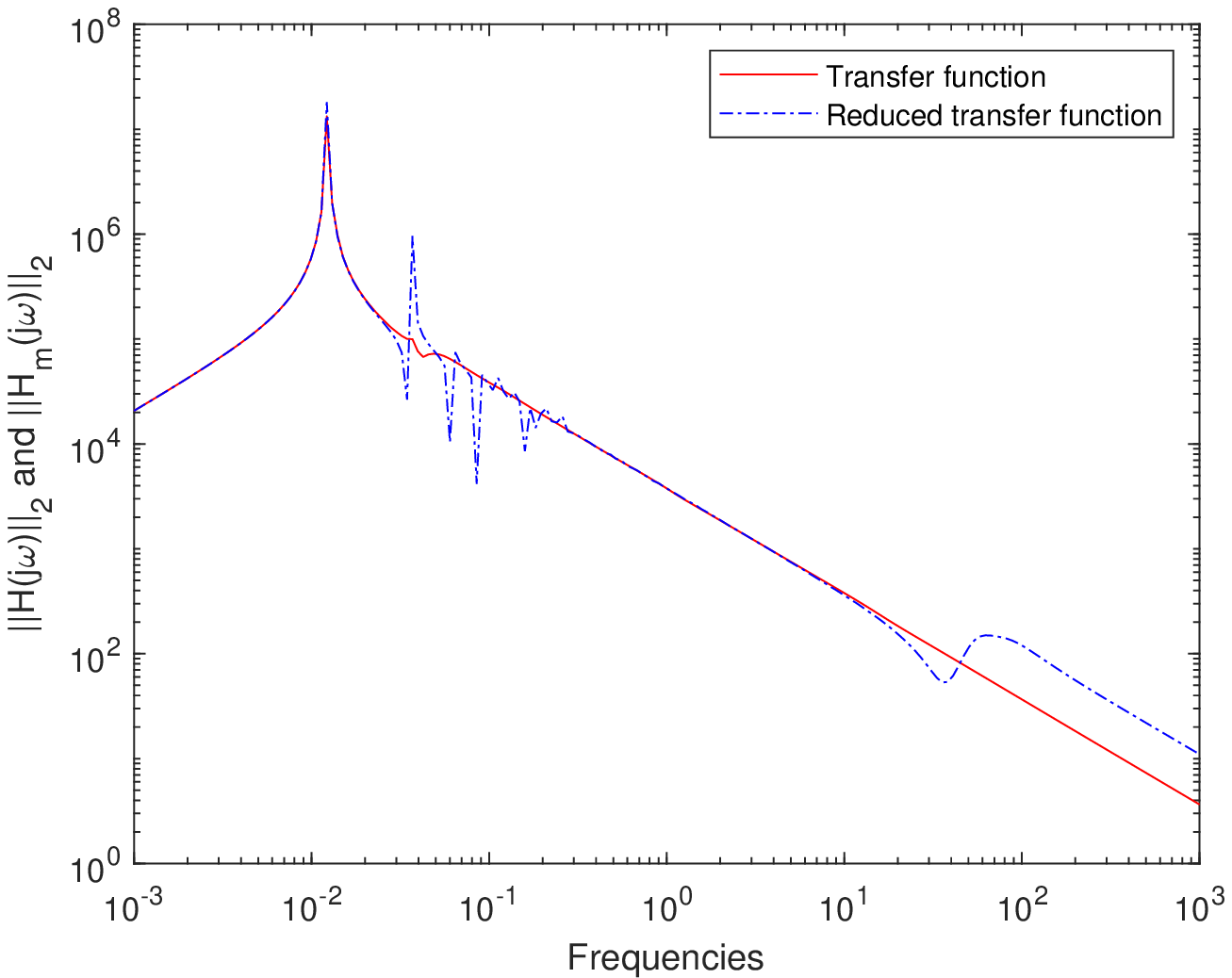}
      \caption{1DBeam-LF5000 mode, m=40.}
     \label{fig04}
   \end{minipage}
\end{figure}
\noindent The plots in  Figures \ref{fig03} and \ref{fig04}, represent the original transfer function  $\|H(j\omega)\|_2$ and the  norm of the reduced transfer function  $\|H_m(j\omega)\|_2$ where we modified the matrices B and C (random matrices) to get a MIMO system with four inputs and four outputs.\\

\noindent The plots in  Figures \ref{fig05} and \ref{fig06}, represent the exact error $\|H(j\omega) - H_m(j\omega)\|_2$ versus the frequencies $\omega \in [10^{-6},\; 10^6]$ of the  1Dbeam-LF10000 model  $2\times n = 39996$ with one input and one output. \\ 
\begin{figure}[!h]
   \begin{minipage}[c]{.46\linewidth}
      \includegraphics[width=1.0\textwidth]{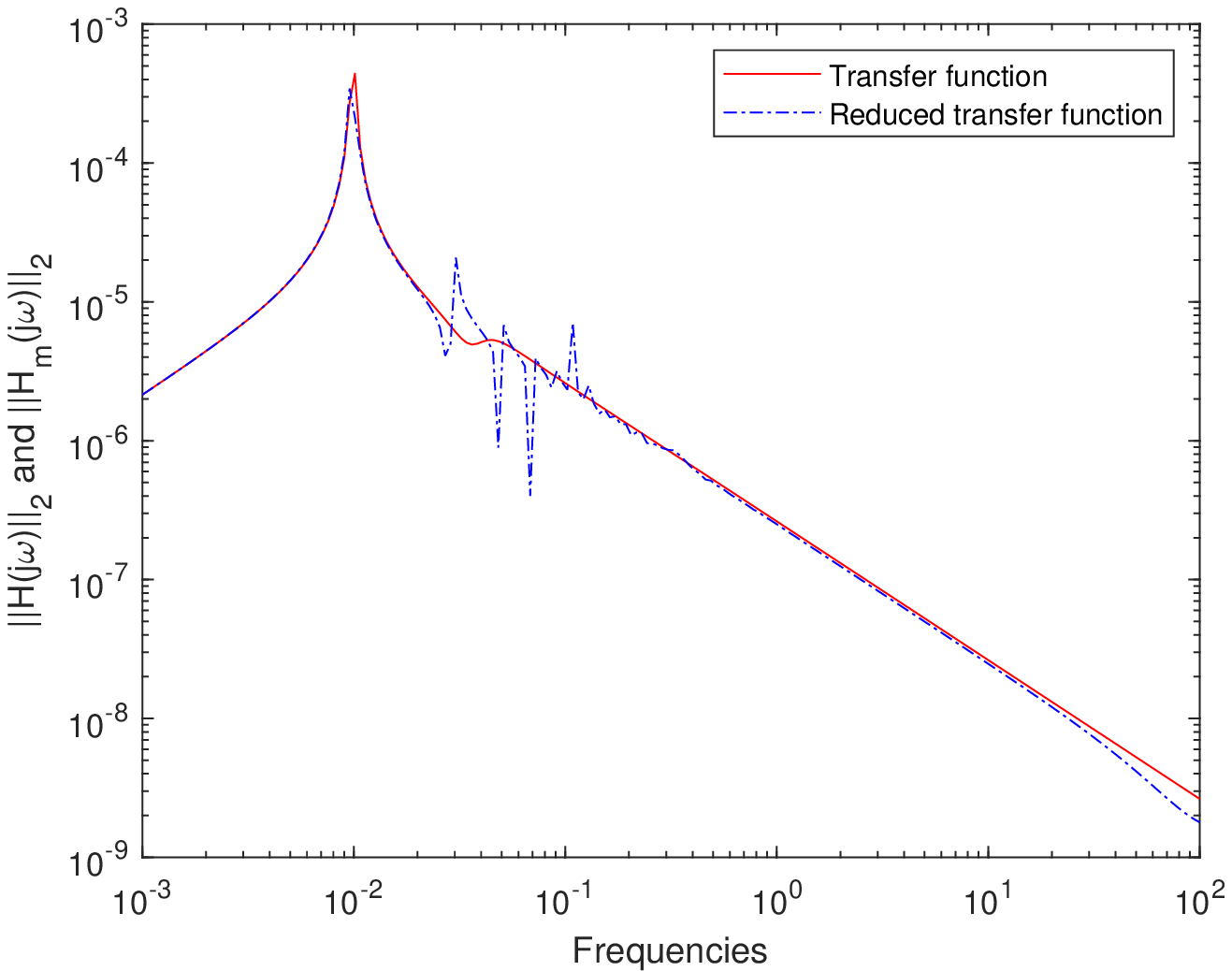}
     \caption{1DBeam-LF10000 model, m=10.}
      \label{fig05}   
   \end{minipage} \hfill
   \begin{minipage}[c]{.46\linewidth}
      \includegraphics[width=1.0\textwidth]{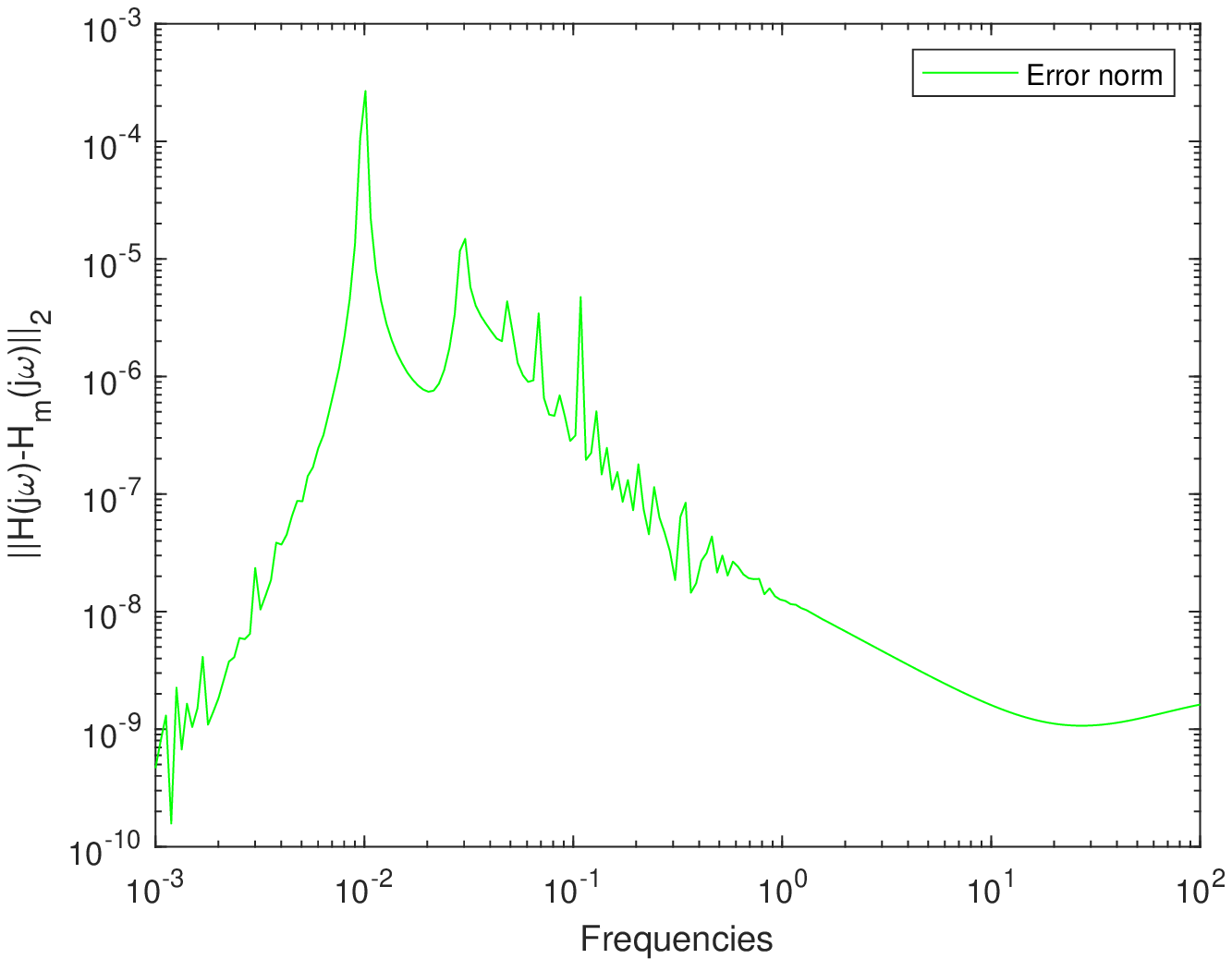}
      \caption{The error norm.}
     \label{fig06}
   \end{minipage}
\end{figure}
 
 \subsection{Example 3: Butterfly Gyroscope}\quad

 \noindent The structural model of the gyroscope has been done in ANSYS (the global leader in engineering simulation) using quadratic tetrahedral elements. The model used here is a simplified one with a coarse mesh as it is designed to test the model reduction approaches. It includes the pure structural mechanics problem only. The load vector is composed from time-varying nodal forces applied at the centers of the excitation electrodes. The Dirichlet boundary conditions have been applied to all degree of freedom of the nodes belonging to the top and bottom surfaces of the frame. This benchmark is also part of the Oberwolfach Collection. It is a second-order model of dimension $n=17361$, (then the matrix  $ \mathcal{A}$ is of size $2\times n = 34722$)  with the matrix $B = C^T$ to get a MIMO system with 12  inputs and 12 outputs.\\

\noindent The plots in Figure \ref{fig07}  represent $\|H(j\omega)\|_2$ and the  norm of the reduced transfer function  $\|H_m(j\omega)\|_2$. Figure \ref{fig08}  represent the exact error $\|H(j\omega) - H_m(j\omega)\|_2$ versus the frequencies. The dimension of the reduced model is $m=40$.  The  execution time was 41.83 seconds with $\mathcal{H}_{\infty}$-err norms equal to $1.73\times10^{-3}$.

\begin{figure}[!h]
   \begin{minipage}[c]{.46\linewidth}
      \includegraphics[width=1.0\textwidth]{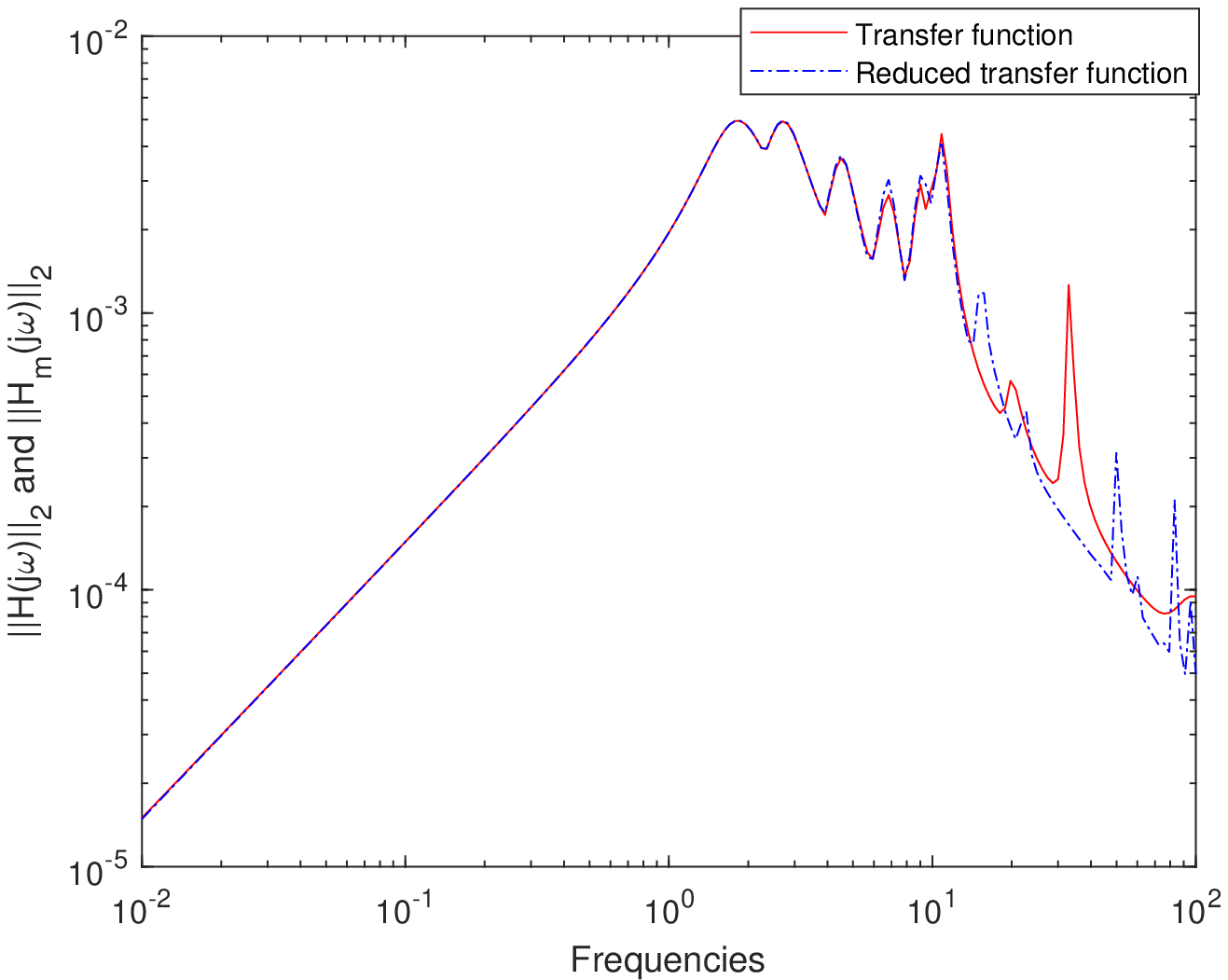}
     \caption{Butterfly model, m=40.}
      \label{fig07}   
   \end{minipage} \hfill
   \begin{minipage}[c]{.46\linewidth}
      \includegraphics[width=1.0\textwidth]{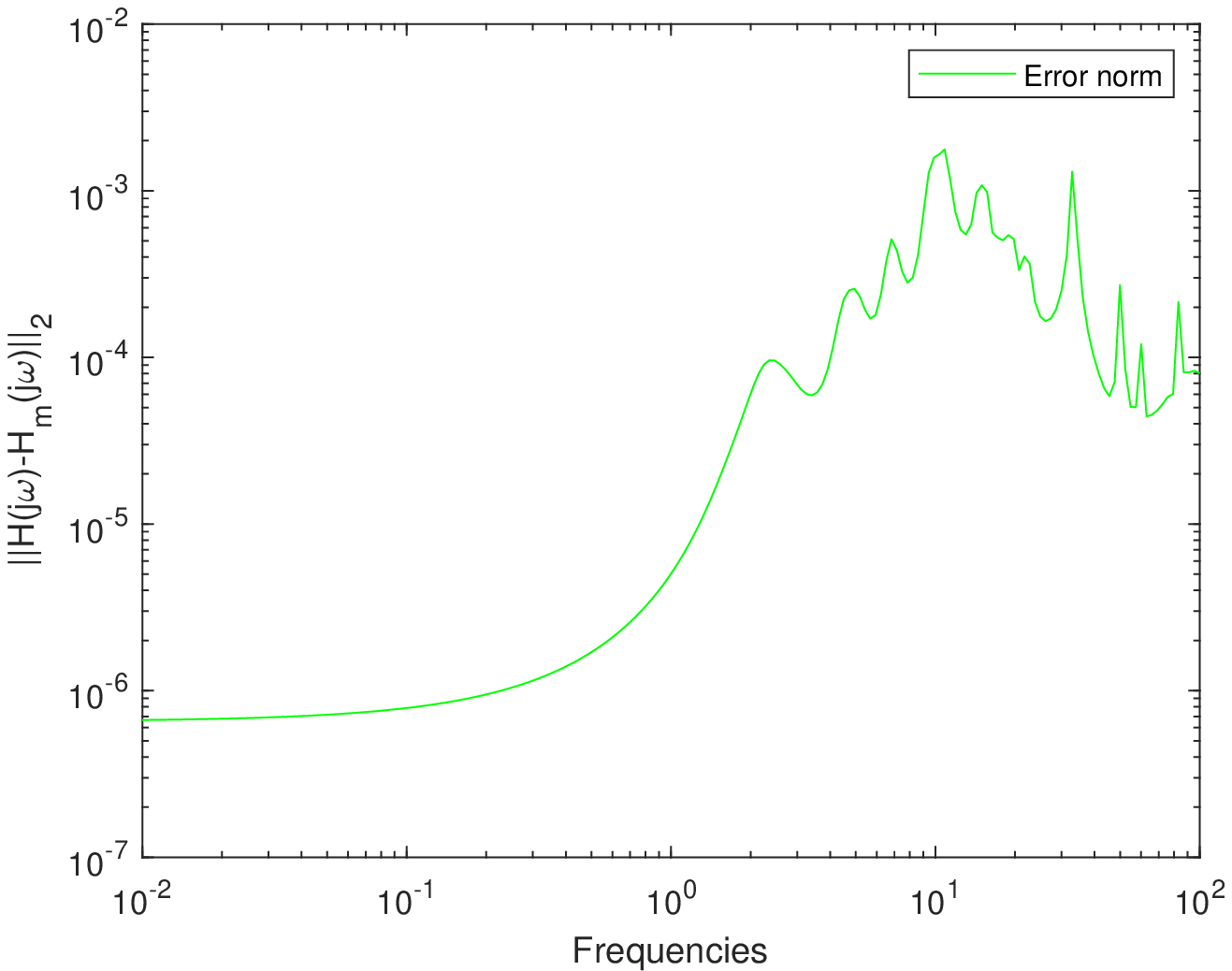}
      \caption{The error norm.}
     \label{fig08}
   \end{minipage}
\end{figure}
	
\noindent The plots in Figure \ref{fig09}  represent $\|H(j\omega)\|_2$ and the  norm of the reduced transfer function  $\|H_m(j\omega)\|_2$. Figure \ref{fig10}  represent the exact error $\|H(j\omega) - H_m(j\omega)\|_2$ versus the frequencies with $m=80$. The  execution time was 96.26 seconds with $\mathcal{H}_{\infty}$-err norms equal to $1.52\times10^{-2}$.

\begin{figure}[!h]
   \begin{minipage}[c]{.46\linewidth}
      \includegraphics[width=1.0\textwidth]{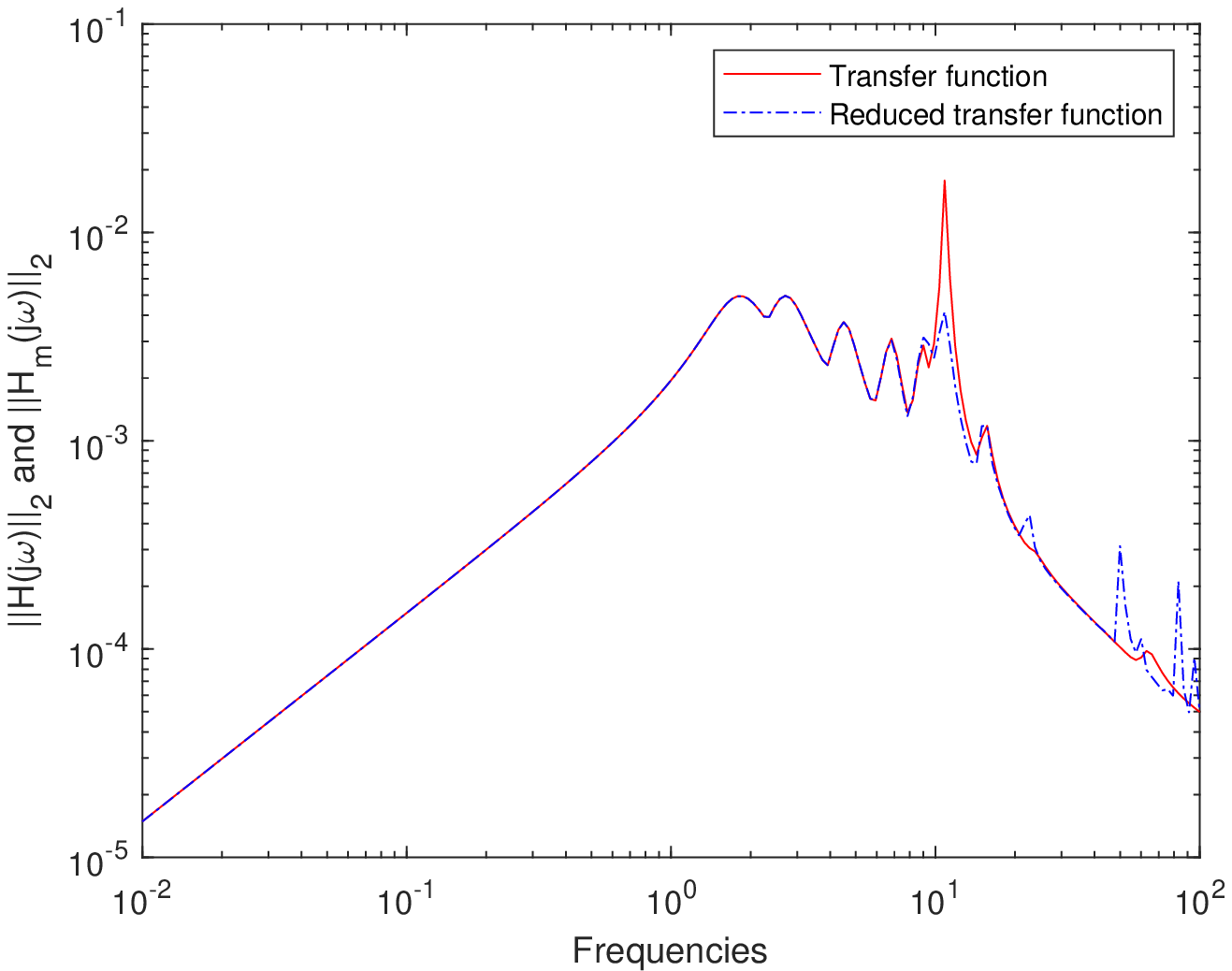}
     \caption{Butterfly model, m=80.}
      \label{fig09}   
   \end{minipage} \hfill
   \begin{minipage}[c]{.46\linewidth}
      \includegraphics[width=1.0\textwidth]{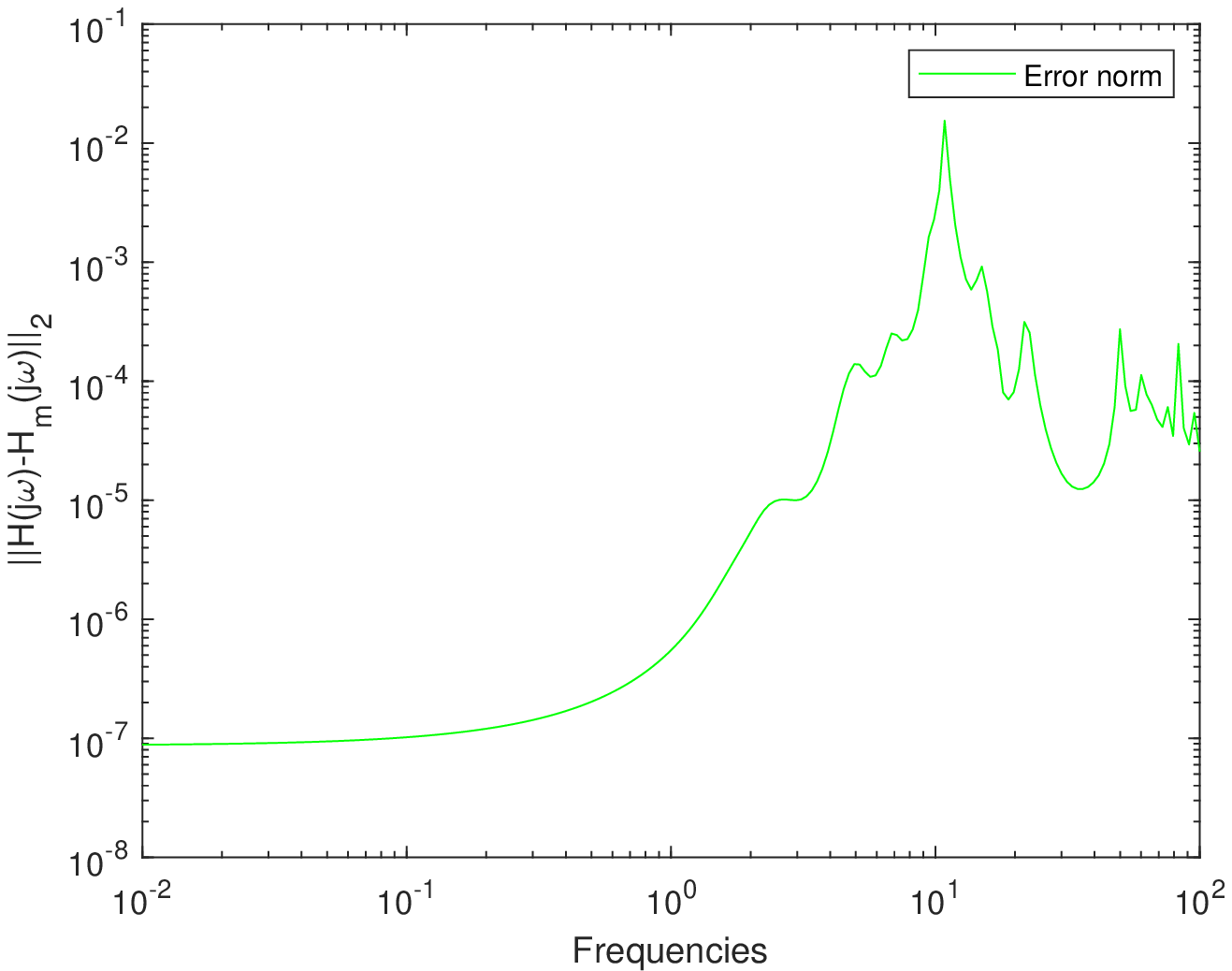}
      \caption{The error norm.}
     \label{fig10}
   \end{minipage}
\end{figure}
\noindent we notice that the tree methods coincide, with an execution time almost the same (TRKSM: 9.26 seconds, ABTL: 9.76 seconds, ABTA: 10.45 seconds)

\vspace{0.3cm}
\newpage
\section{Conclusion}
In the present paper, we proposed a new approach based on block  tangential Krylov subspaces to
compute low rank approximation of  large-scale first and second order  dynamical systems with multiple inputs and multiple outputs (MIMO). The method  constructs sequences of orthogonal blocks from matrix tangential Krylov subspaces using the block Lanczos-type approach. The interpolation shifts  and the tangential directions are selected in an adaptive way by maximizing the residual norms. We gave some new algebraic properties and compared our algorithms with well knowing methods  to show  the effectiveness of this latter.


\vspace{0.3cm}
\begin{thebibliography}{1}

\bibitem{B004} 
A. C. Antoulas, Approximation of large-scale dynamical systems, Adv. Des. Contr., SIAM 2005.

\bibitem{B0041} 
A. C. Antoulas, C. A. Beattie and S. Gugercin, Interpolatory model reduction of large-scale dynamical systems. Effic. Mod. Contr. Syst., (2010), 3--58.

\bibitem{B23} 
Z. Bai, Krylov subspace techniques for reduced-order modeling of large scale dynamical systems. Appl. Numer. Math., 43(2002), 9--44.

\bibitem{B26} 
C. A. Beattie, S. Gugercin, Krylov-based model reduction of second-order systems with proportional damping. Proceedings of the $44^{th}$ IEEE Conference on Decision and Control, (2005),  2278--2283.

\bibitem{B006}
P. Benner, J. LI and T. Penzl, Numerical solution of large Lyapunov equations, Riccati equations, and linear-quadratic optimal control problems. Numer. Lin. Alg. Appl., 15(2008), 755--777.

\bibitem{B27} 
Y. Chahlaoui, K. A. Gallivan, A. Vandendorpe and P. Van Dooren, Model reduction of second-order system. Comput. Sci. Eng., 45(2005), 149--172.

\bibitem{B28} 
Y. Chahlaoui, D. Lemonnier, A. Vandendorpe and P. Van Dooren, Second-order balanced truncation. Lin. 
Alg. Appli., 415(2006), 373--384.

\bibitem{B32} 
B. N. Datta, Large-scale matrix computations in Control. Appl. Numer. Math., 30(1999), 53--63.

\bibitem{B33}
B. N. Datta, Krylov subspace methods for large-scale matrix problems in control. Futu. Gener. Comput. Syst.,  19(2003), 1253--1263.

\bibitem{B008} 
V. Druskin, V. Simoncini and M. Zaslavsky, Adaptive tangential interpolation in rational Krylov subspaces for MIMO dynamical systems. SIAM. J. Matrix Anal. and Appl., 35(2014), 476--498.

\bibitem{B009} 
V. Druskin, C. Lieberman and M. Zaslavsky, On adaptive choice of shifts in rational
Krylov subspace reduction of evolutionary problems. SIAM J. Sci. Comput., 32(2010), 2485--2496.

\bibitem{B010} 
V. Druskin, V. Simoncini, Adaptive rational Krylov subspaces for large-scale dynamical
systems. Syst. Contr. Lett., 60(2011), 546--560.

\bibitem{B005} 
L. Fortuna, G. Nunnari and A. Gallo, Model order reduction techniques with applications in
electrical engineering. Springer-Verlag London, 1992.

\bibitem{B34}
M. Frangos, I.M. Jaimoukha, Adaptive rational interpolation: Arnoldi and Lanczos-like equations. Eur. J. Control., 14(2008), 342--354.


\bibitem{B002} 
K. Glover, All optimal Hankel-norm approximations of linear multivariable systems and their L, $\infty$-error bounds. Inter. Jour. Cont., 39(1984), 1115--1193.   

\bibitem{B31}
K. Glover, D. J. N. Limebeer, J. C. Doyle, E. M. Kasenally and M. G. Safonov, A characterization of all solutions to the four block general distance problem. SIAM J. Contr. Optim., 29(1991), 283--324.


\bibitem{B9} 
E. Grimme, Krylov projection methods for model reduction. Ph.D. thesis, Coordinated Science Laboratory, University of Illinois at Urbana--Champaign, 1997.

\bibitem{B007}
S. Gugercin, A.C. Antoulas, A survey of model reduction by balanced truncation and some
new results. Inter. Jour. Cont., 77(2004), 748--766.

\bibitem{B38} 
S. Gugercin, A.C. Antoulas and C. Beattie, A rational Krylov iteration for optimal $\mathcal{H}_2$ model reduction. J. Comp. Appl. Math., 53(2006), 1665--1667.

\bibitem{B36}
M. Heyouni, K. Jbilou, Matrix Krylov subspace methods for large scale model reduction problems. App. Math. Comput., 181(2006), 1215--1228.

\bibitem{B35}
 I. M. Jaimoukha, E. M. Kasenally, Krylov subspace methods for solving large Lyapunov equations. SIAM J. Matrix Anal. Appl., 31(1994), 227--251.

\bibitem{B003} 
B. C. Moore, Principal component analysis in linear systems: controllability, observability and model reduction. IEEE Trans. Automatic Contr., 26(1981), 17--32.  

\bibitem{B001} 
C. T. Mullis, R. A. Roberts,  Round off noise in digital filters: Frequency trans-formations and invariants. IEEE Trans. Acoust. Speec Signal Process., 24(1976), 538--550.   

\bibitem{B011} 
T. Penzl, LYAPACK matlab toolbox for large Lyapunov and Riccati equations, model reduction problems, and linear-quadratic optimal control problems, http://www.tuchemintz.de/sfb393/lyapack.

\bibitem{B23} 
A. Preumont, Vibration control of active structures. An introduction, second ed., Kluwer, Dordrecht, 2002.

\bibitem{B24} 
M. F. Rubinstein, Structural systems-statics, dynamics and stability. Prentice Hall, Inc, 1970.

\bibitem{B30} 
M. G. Safonov and R. Y. Chiang, A Schur method for balanced-truncation model reduction. IEEE Trans. Automat. Contr., 34(1989),  729--733.

\bibitem{B22} 
P. Van Dooren, K. A. Gallivan and P. Absil, $H_2$-optimal model reduction with higher order poles. SIAM J.  Matr. Analy.  Appli., 31(2010), 2738--2753.
  
\bibitem{B25} 
W. Weaver, P. Johnston, Structural dynamics by finite elements. Prentice Hall, Inc, 1987.

\bibitem{B37} 
W. Weaver, S. P. Timoshenko, D. H. Young, Vibration problems in engineering, $5^{th}$ ed., Wiley, 1990.


\end{thebibliography}
\end{document}